\documentclass[11pt, a4paper]{amsart}
\usepackage[utf8]{inputenc}
\usepackage[english]{babel}
\usepackage{enumerate}
\usepackage{color}
\usepackage{stmaryrd}
\usepackage[text={14.06cm,20.84cm},centering]{geometry}
\usepackage{marginnote} \usepackage{framed}
\usepackage{siunitx}
\usepackage{xcolor}
\usepackage{esint,amsmath,amssymb}
\usepackage{graphicx}
\usepackage{mathtools}
\usepackage[colorlinks=true, pdfstartview=FitV, linkcolor=blue, citecolor=blue, urlcolor=blue,pagebackref=false]{hyperref}
\usepackage{microtype}

\usepackage{bm}
\usepackage{scalerel} 
\usepackage{mathrsfs}
\usepackage{tikz}
\usepackage[all]{xy} \xyoption{arc} \xyoption{color}
\usepackage{epsfig}
\usepackage{amsbsy}
\usepackage[font={footnotesize}]{caption}

\usepackage{enumitem}

\usepackage{comment}

\newcommand{\eps}{\varepsilon}

\newtheorem{proposition}{Proposition}
\newtheorem{theorem}[proposition]{Theorem}

\newtheorem{corollary}[proposition]{Corollary}

\theoremstyle{remark}
\newtheorem{remark}[proposition]{Remark}

\theoremstyle{definition}
\newtheorem{definition}[proposition]{Definition}

\numberwithin{equation}{section}
\numberwithin{proposition}{section}

\begin{document}

\title[Asymmetric equilibrium configurations of a FSI problem]{Asymmetric equilibrium configurations\\ of a body immersed in a 2D laminar flow}
\author[Edoardo Bocchi and Filippo Gazzola]{Edoardo Bocchi$^\ast$ and Filippo Gazzola$^\ast$}\thanks{$^\ast$Dipartimento di Matematica, Politecnico di Milano, Italy - MUR Excellence Department 2023-2027. Email: edoardo.bocchi@polimi.it - filippo.gazzola@polimi.it}

\begin{abstract}
We study the equilibrium configurations of a possibly asymmetric fluid-structure-interaction problem. The fluid is confined in a bounded planar channel and
is governed by the stationary Navier-Stokes equations with laminar inflow and outflow. A body is immersed in the channel and is subject to both the lift force from the fluid and to some external elastic force. Asymmetry, which is
motivated by natural models, and the possibly non-vanishing velocity of the fluid on the boundary of the channel require the introduction of suitable assumptions to prevent collisions of the body with the boundary. With these assumptions at hand, we prove that for sufficiently small
inflow/outflow there exists a unique equilibrium configuration. Only if the inflow, the outflow and the body are all symmetric, the configuration is also symmetric. A model application is also discussed.\\
\textbf{Mathematics Subject Classification:}
35Q35, 
76D05, 
74F10. 
\end{abstract}

\maketitle

\section{Introduction}\label{intro}
Let $L>H>0$ and consider the rectangle $R=(-L,L)\times(-H,H)$. Let $B\subset R$ be a closed smooth domain having barycenter at the origin
$(x_1,x_2)=(0,0)$ and such that $\mathrm{diam}(B)\ll L,H$. We study the behavior of a stationary laminar (horizontal) fluid flow
going through $R$ and filling the domain $\Omega_h= R\setminus B_h$, where $B_h=B+he_2$ for some $h$ (a vertical translation of $B$),
see Figure \ref{domain}. Note that $B_0=B$.
\begin{figure}[!h]
	\begin{center}
		\includegraphics[width=\textwidth]{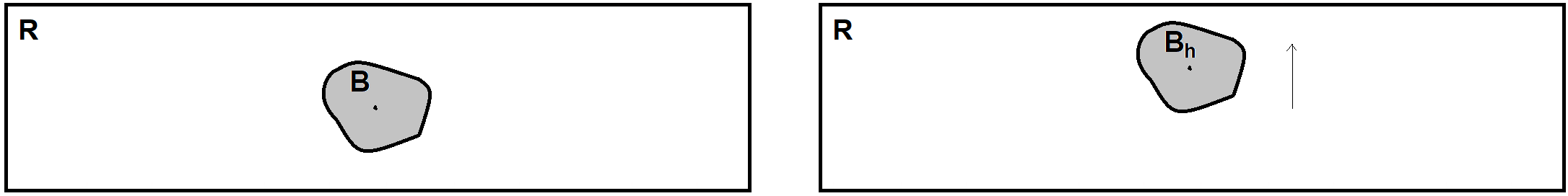}
		\caption{The rectangle $R$ and the body $B$ with its vertical displacements $B_h$.}\label{domain}
	\end{center}
\end{figure}

The fluid is governed by the stationary 2D Navier-Stokes equations
\begin{equation}\label{NS}-\mu\Delta u+u\cdot \nabla u+ \nabla p = 0,\quad\nabla \cdot u =0\quad\mbox{in}\quad \Omega_h,\end{equation}
complemented with inhomogeneous Dirichlet boundary conditions on $\partial\Omega_h=\partial B_h\cup \partial R$, see \eqref{BVprob} below.
Here, $\mu>0$ is the kinematic viscosity, $u$ is the velocity vector field, $p$ is the scalar pressure.\par
The body $B$ is subject to two vertical forces. The first force (the lift) is due to the fluid flow and tends to move $B$
away from its original position $B_0$, it is expressed through a boundary integral over $\partial B$, see \eqref{lift} below.
The second force is mechanical (elastic) and acts as a restoring force tending to maintain
$B$ in $B_0$. When there is no inflow/ouflow, the body is only subject to the restoring force and remains in $B_0$ which
is the unique equilibrium position. But, as soon as there is a fluid flow, these two forces start competing and one may wonder if the
body remains in $B_0$ or, at least, if the equilibrium position remains unique.\par
We show that, if the inflow/ouflow is sufficiently small, then the equilibrium position of $B$ remains unique and coincides with $B_h$
for some $h$ close to zero. We point out that, contrary to \cite{BonGalGaz20,gazzolapatriarca,GazSpe20}, {\it we make no symmetry assumptions neither
on $B$ nor on the laminar inflow/outflow}. Therefore, not only the overall configuration will be asymmetric but also some of the techniques developed in these papers do not work and $B_h$ may be different
from $B_0$. The motivation for studying
asymmetric configurations comes from nature. Only very few bodies are perfectly symmetric and most fluid
flows, although laminar in the horizontal direction, are asymmetric in the vertical direction: think of an horizontal wind depending on the altitude or the water flow in a river
depending on the distance from the banks. Figure \ref{windstorm} shows two front waves in sandstorms that have
no vertical symmetry although the wind is (almost) horizontally laminar.
\begin{figure}[!h]
\begin{center}
\includegraphics[width=0.49\textwidth]{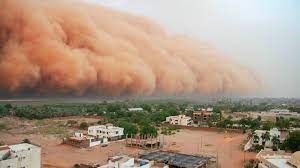}\quad\includegraphics[width=0.46\textwidth]{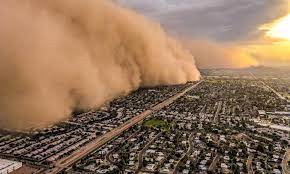}
\caption{Front wave of two wind storms.}\label{windstorm}
\end{center}
\end{figure}

In Section \ref{fluid-section} we give a detailed description of our model and we prove that, for small Reynolds numbers,
the Navier-Stokes equations are uniquely solvable in any $\Omega_h$, see Theorem \ref{strong-exiuni}. The related a priori bounds
depend on $h$, and this is one crucial difference compared to the (symmetric) Poiseuille inflow/outflow considered in \cite{BonGalGaz20}.
It is well-known \cite{Galdi-steady} that
to solve inhomogeneous Dirichlet problems for the Navier-Stokes equations, one needs to find a solenoidal extension of the boundary
data and to transform the original problem in an homogeneous Dirichlet problem with an additional source term. For the existence issue,
one can use the classical Hopf extension, but there are infinitely many other possible choices for the solenoidal extension. One of them, introduced
in \cite{HoLeal74}, was used in \cite{BonGalGaz20} to write the lift force as a volume
integral by means of the solution of an auxiliary Stokes problem. For asymmetric flows, the same solenoidal extension does not allow
to estimate all the boundary terms and, in order to obtain refined bounds for the solution to the Navier-Stokes equations in $\Omega_h$,
we build a new explicit solenoidal extension that also plays a fundamental role in the analysis of the subsequent fluid-structure-interaction (FSI) problem.\par
The main physical interest in FSI problems is to determine the $\omega$-limit of the associated evolution equations
because this allows to forecast the long-time behavior of the structure. Since the evolution Navier-Stokes equations are dissipative,
one is led to investigate if the global attractor exists, see \cite{gazpatpat,clara}: the main difficulty is that the corresponding
phase space is time-dependent and semigroup theory does not apply. The global attractor contains stationary
solutions of the evolution FSI problem that we call equilibrium configurations, which are investigated in the present work.\par
In Section \ref{sec-FSI} we introduce the lift force and the restoring force and we set up the steady-state FSI problem. Our main result, namely Theorem
\ref{main-theo}, states that, for small Reynolds numbers, the equilibrium position is unique and may differ from $B_0$. By exploiting the strength of the restoring force,  uniqueness for the FSI problem is obtained without assuming uniqueness for \eqref{NS}.
To prove this result, we need some bounds on the lift force in proximity of collisions of $B_h$ with $\partial R$: these bounds are
collected in Theorem \ref{lift-theo} and proved in Section \ref{sec-proofTheos} by using the very same solenoidal extension introduced in
Section \ref{fluid-section}. The remaining part of the proof of Theorem \ref{main-theo} is divided in two steps. In Subsection
\ref{cont-mono-phi} we prove some properties of the global force exerted on the body $B$. These properties are then used in
Subsection \ref{conclu-proof} to complete the proof by means of an implicit function argument, combined with some delicate bounds
involving derivatives of moving boundary integrals. We emphasize that for our FSI problem we cannot use the explicit expression of the lift derivative as
in \cite{Piro74} because the displacements $B_h$ within $R$ do not follow the normal of $\partial B_h$, in particular if
$\partial B_h$ contains some vertical segments. Instead, based on the general approach introduced in \cite{BelFCLemSim97} (see also the previous work \cite{MurSim74}), we compute with high precision
the lift variation with respect to the vertical displacement parameter $h$ of $B_h$ by acting directly on the strong form of the FSI problem.\par
Section \ref{FSIsymm} contains the symmetric version of Theorem \ref{main-theo}, see Theorem \ref{symmtheo} which states that, under
symmetry assumptions on the inflow/outflow and on $B$, for small Reynolds numbers the equilibrium position is unique and coincides with $B_0$.
This extends former results in \cite{BonGalGaz20,gazzolapatriarca,GazSpe20} to a wider class of symmetric frameworks.\par
As an application of our results, in Section \ref{suspension} we consider a model where $B_h$ represents the cross-section of the deck of a
suspension bridge \cite{gazzbook}, while $\Omega_h$ is filled by the air and represents either a virtual box
around the deck or a wind tunnel around a scaled model of the bridge. Since the deck may have a nonsmooth boundary,
we also explain how to extend our results to the case where $B$ is merely Lipschitz.

\section{Fluid boundary-value problem}\label{fluid-section}

Let $R$ and $B$ be as in Section \ref{intro} (Figure \ref{domain}) with
\begin{equation}\label{B-W2infty}
	B \mbox{ of class } W^{2,\infty}.
\end{equation}
On the one hand, \eqref{B-W2infty} ensures the regularity $(u,p)\in H^2(\Omega_h)\times H^1(\Omega_h)$ for the solutions to \eqref{NS}, see \cite[Theorem 2.1]{MurSim74} and  Theorem \ref{strong-exiuni} below. On the other hand, in engineering applications $B$ is usually a polygon with rounded corners, see Section \ref{suspension}, which belongs to $W^{2,\infty}$ but not to $C^2$. Let
\begin{equation}\label{def-delta}\delta_b:= -\min_{(x_1,x_2)\in \partial B}x_2>0, \quad \delta_t:= \max_{(x_1,x_2)\in \partial B}x_2>0, \quad \tau:= \max_{(x_1,x_2)\in \partial B}|x_1|.\end{equation}
Since we consider vertical displacements $B_h$ within $R$, we have $h\in(-H+\delta_b,H-\delta_t)$ and
$B_h \subset [-\tau,\tau]\times  [h-\delta_b, h +\delta_t]$ for any such $h$. Then, $\partial\Omega_h=\partial B_h\cup\partial R$.
The bottom and top parts of $\partial R$ are respectively
$$\Gamma_b=[-L,L]\times\{-H\} \quad \mbox{and} \quad\Gamma_t=[-L,L]\times\{H\},$$ while its lateral left and right parts are, respectively,
$$\Gamma_l=\{-L\}\times[-H,H]\quad \mbox{and} \quad\Gamma_r=\{L\}\times[-H,H].$$
Let $V_{\rm in}, V_{\rm out}\in W^{2,\infty}(-H,H)\subset  C^0[-H,H]$ satisfy
\begin{equation}\begin{aligned}\label{match-conds}
V_{\rm in}(-H)&=V_{\rm out}(-H)=0, \quad V_{\rm in}(H)=V_{\rm out}(H)=U\geq 0, \\[2pt]& \int_{-H}^H V_{\rm in}(x_2) dx_2= 	\int_{-H}^H V_{\rm out}(x_2) dx_2.\end{aligned}
\end{equation}
For some $\lambda\ge0$, we consider the boundary-value problem
\begin{equation}\label{BVprob}
\begin{array}{cc}
-\mu\Delta u + u\cdot \nabla u + \nabla p = 0,\qquad\nabla \cdot u =0\quad\mbox{in}\quad \Omega_h,\\[2pt]
u_{|_{\partial B_h}}\!=\!u_{|_{\Gamma_b}}\!=0, \quad u_{|_{\Gamma_t}}\!=\lambda Ue_1,\quad
u_{|_{\Gamma_l}}\!=\lambda V_{\rm in}(x_2)e_1,\quad u_{|_{\Gamma_r}}\!=\lambda V_{\rm out}(x_2)e_1.\\[2pt]
\end{array}
\end{equation}
Note that $u_{|_{\partial R}}\in C^0(\partial R)$ and \eqref{match-conds}-\eqref{BVprob} are compatible with the Divergence Theorem.
The role of $\lambda\geq 0$ in the boundary conditions is to measure with a unique parameter
the strength of both the inflow and the outflow. Hence, $\lambda\asymp\rm{Re}$ where $\rm{Re}$ is the Reynolds number.

\begin{definition}
	We say that $(u,p)\in H^2(\Omega_h)\times H^1(\Omega_h)$ is a strong solution to \eqref{BVprob} if the differential equations are satisfied
	a.e.\ in $\Omega_h$ and the boundary conditions are satisfied as restrictions (recall that $H^2(\Omega_h)\subset C^0(\overline{\Omega_h})$).
\end{definition}

We now state an apparently classical existence and uniqueness result which, however, has some novelties. First, since the domain $\Omega_h$
is only Lipschitzian, the regularity of the solution is obtained through a geometric reflection. More important, the explicit
upper bound for the blow-up of the $H^1$-norm of the unique solution to \eqref{BVprob} in proximity of collision: when $B$ approaches $\Gamma_t$
the norm remains bounded while when $B$ approaches $\Gamma_b$ we estimate its blow-up. This refined bound requires the construction of a suitable
solenoidal extension of the boundary data. Note that, up to normalization, we can reduce to the cases where
\begin{equation}\label{U01}
U\in\{0,1\}.\vspace{0.5em}\end{equation}
In order to state the result, we define the distances of the body $B_h$
to $\Gamma_b$ and $\Gamma_t$ respectively by
\begin{equation}\label{parameters}
\eps_b(h):= H-\delta_b + h, \qquad\eps_t(h):= H-\delta_t -h.
\end{equation}
Hence, $0<\eps_b(h),\eps_t(h) \leq 2H -\delta_b -\delta_t$ for any $h\in(-H+\delta_b, H-\delta_t)$.
Throughout the paper, any (positive) constant depending only on $\mu$, $B_0$, $L$, $H$ will be denoted by $C$ and, when it depends also on $h$, by $C_h$. We may now state

\begin{theorem}\label{strong-exiuni}
	Let $h\in(-H+\delta_b,H-\delta_t)$ and assume \eqref{match-conds} with \eqref{U01}.
	Then \eqref{BVprob} admits a strong solution $(u,p)$ for any $\lambda\geq0$ and there exists $\Lambda=\Lambda(h)>0$
 such that the solution
is unique if $\lambda\in [0,\Lambda(h))$;  if $U=0$, $\Lambda(h)$ can be chosen independent of $h$, \textit{i.e.} $\Lambda(h)\equiv\Lambda>0$. Moreover, there exist $C>0$ and $C_h>0$ such that
	the unique solution (when $\lambda<\Lambda(h)$) satisfies
	\begin{align}\label{energy-est}
&	\| u\|_{H^1(\Omega_h)}  \leq C	(1+U(\eps_t(h))^{-3/2})\lambda,\\[2pt]
\label{reg-est}
&	\|u\|_{H^2(\Omega_h)} + \|p\|_{H^1(\Omega_h)} \leq C_h \lambda.
	\end{align}
	A priori bounds such as \eqref{energy-est} and \eqref{reg-est} are available for any $\lambda \geq 0$ and any strong solution of \eqref{BVprob}, but
	with different powers of $\lambda$.
\end{theorem}
Before giving the proof, let us explain qualitatively the main differences between the cases $U=0$ and $U=1$. 
	For $U=0$, the a priori bound \eqref{energy-est} is \textit{independent} of $h$, so that the graph of $\Lambda(h)$ looks like Figure \ref{Lambdone} (left). For $U=1$, \eqref{energy-est} depends on $h$ and $\Lambda(h)$ itself may depend on $h$, see Figure \ref{Lambdone} (right) and \eqref{smallness} below.
		\begin{figure}[!h]
		\includegraphics[height=4cm]{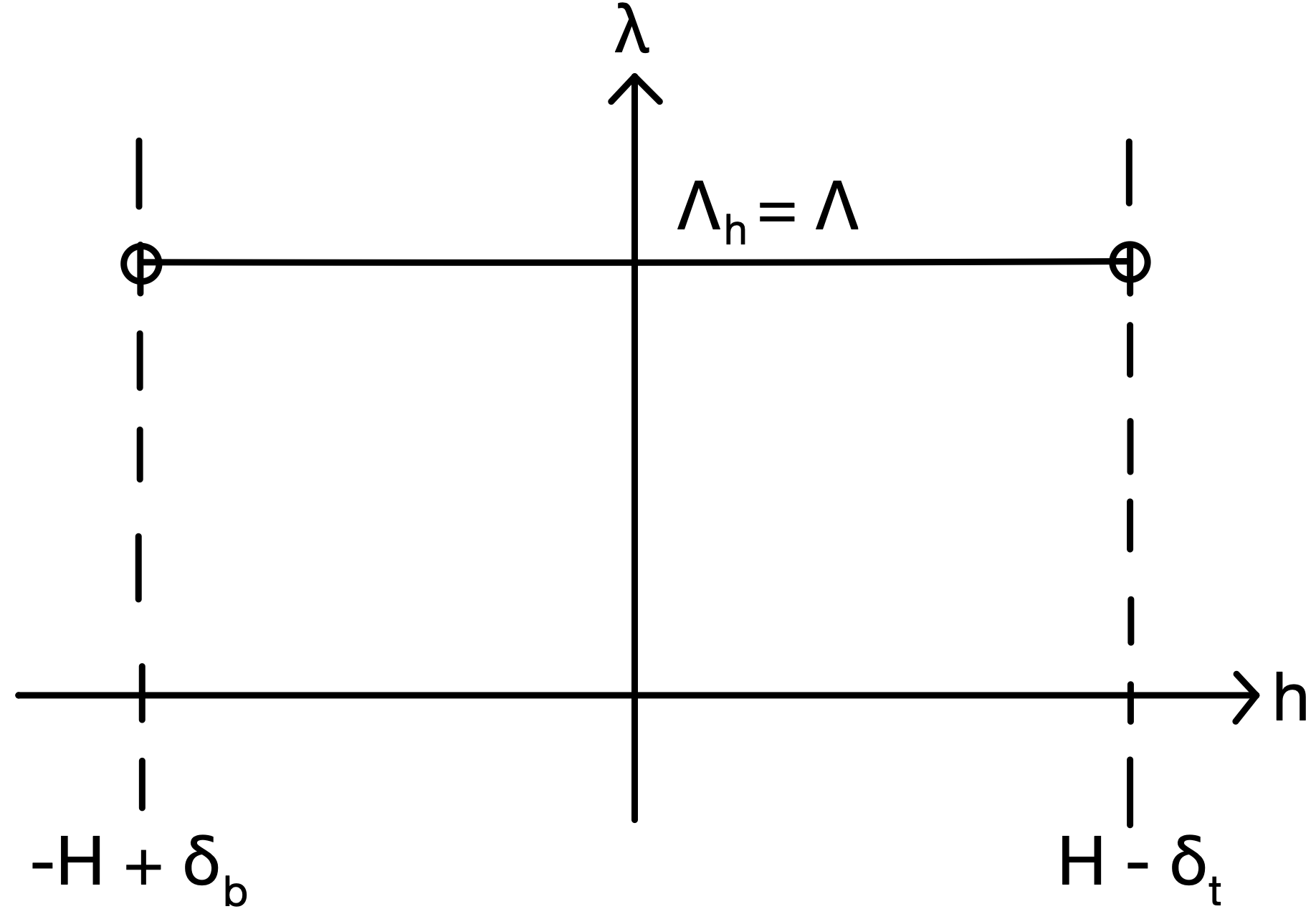} \qquad
		\includegraphics[height=4cm]{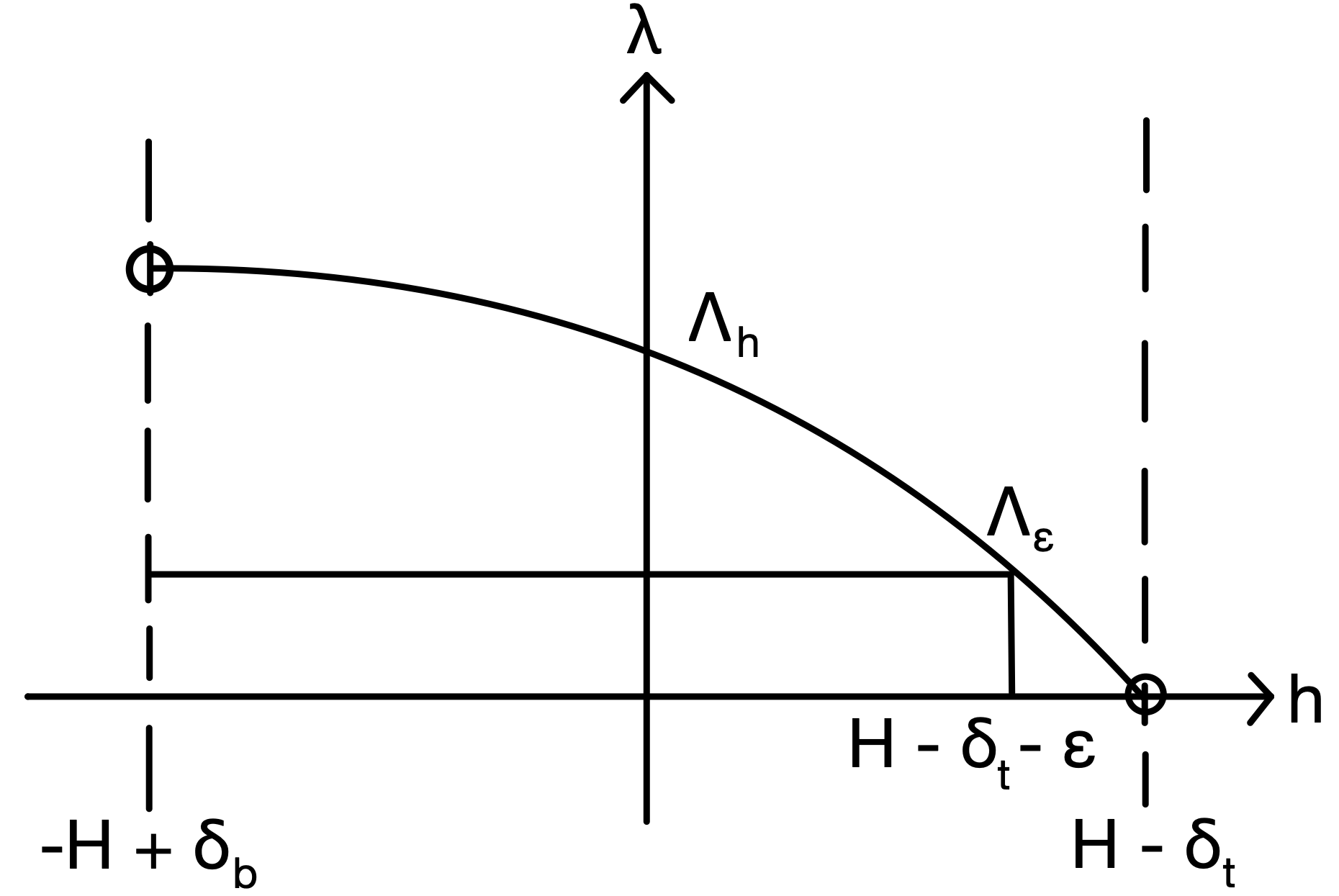}
		\caption{Qualitative behavior of $\Lambda=\Lambda(h)$ for $U=0$ (left) and $U=1$ (right).}
		\label{Lambdone}
	\end{figure}

\begin{proof}
	{\it Existence of weak solutions.} For later use, we first define weak solution for the forced Navier-Stokes equations
	\begin{equation}\label{forcedNS}
-\mu\Delta u + u\cdot \nabla u + \nabla p = f,\qquad\nabla \cdot u =0\quad\mbox{in}\quad \Omega_h,
	\end{equation}
	which reduces to \eqref{BVprob} when $f=0$. We say that $u\in H^1(\Omega_h)$ is a weak solution to \eqref{forcedNS} with $f\in L^2(\Omega_h)$  if
	$u$ is a solenoidal vector field satisfying the boundary conditions in the trace sense and
	\begin{equation}\label{weak-BV}
	\mu\int_{\Omega_h} \nabla u : \nabla \varphi +\int_{\Omega_h} u \cdot \nabla u \cdot \varphi =\int_{\Omega_h}f\cdot \varphi\\[2pt]
	\end{equation}
	for all $\varphi\in W(\Omega_h):=\{\varphi\in H^1_0(\Omega_h) : \nabla \cdot \varphi=0 \ \mbox{a.e. in} \ \Omega_h\}$. For any weak solution $u$, there exists a unique associated $p\in L^2_0(\Omega_h)$ (\textit{i.e.} with zero mean value), satisfying
	\begin{equation}\label{weak-pressure}
	\mu\int_{\Omega_h}\nabla u : \nabla \psi +\int_{\Omega_h} u\cdot \nabla u\cdot \psi -\int_{\Omega_h} p \nabla \cdot \psi =\int_{\Omega_h}f\cdot \psi\\[2pt]
	\end{equation}for all $\psi\in H^1_0(\Omega_h)$ (Lemma IX.1.2, \cite{Galdi-steady}).
	In \eqref{solext} below, we introduce an ad-hoc solenoidal extension matching our geometric framework which is not optimal for our current purpose. This is why we use here the well-known Hopf's extension $s$ that reduces the effect of the nonlinearity and allows to prove existence for any $\lambda \geq 0$. Hence, we recast \eqref{BVprob} as \eqref{forcedNS} with homogeneous boundary conditions, namely
	\begin{equation}\label{homo-prob}
	-\mu\Delta v + v\cdot \nabla v +\nabla p  =f,\quad
	\nabla \cdot v =0\quad\mbox{in}\quad \Omega_h,\qquad v_{|_{\partial \Omega_h}}=0,\\[2pt]
	\end{equation}
	where $f=\mu\Delta s -s\cdot \nabla v - v\cdot \nabla s-s\cdot \nabla s $. Then there exists $v\in W(\Omega_h)$ satisfying \eqref{weak-BV} for any $\lambda\geq 0$ (Theorem IX.4.1, \cite{Galdi-steady}). This is equivalent to say that the vector field $u=v+s\in H^1(\Omega_h)$ and the associated pressure $p\in L^2(\Omega_h)$ satisfy \eqref{weak-BV}-\eqref{weak-pressure} with $f=0$. Moreover, $\nabla \cdot u=0$, $u_{|_{\partial \Omega_h}}= s_{|_{\partial \Omega_h}}$ and
	\begin{equation}\label{energy-est-gen}\begin{aligned}
	\|u\|_{H^1(\Omega_h)}&\leq C(\|\nabla v \|_{L^2(\Omega_h)}+\|s\|_{H^1(\Omega_h)})\\[2pt]& \leq C((1+ \tfrac{1}{\mu})\|s\|_{H^1(\Omega_h)}+\tfrac{1}{\mu}\|s\|^2_{H^1(\Omega_h)})\leq C_h(\lambda + \lambda^2)\, ,\\[2pt]
	\end{aligned}
	\end{equation}
	\begin{equation}\label{L2pressure-est}
	\|p\|_{L^2(\Omega_h)}\leq C (\mu \| u\|_{H^1(\Omega_h)}  + \| u\|^2_{H^1(\Omega_h)})\leq C_h (\lambda+\lambda^4).\\[2pt]
	\end{equation}
In these bounds and the ones below we only emphasize the smallest and largest powers of $\lambda$, as for any polynomial. These bounds are not part of the statement but they will be used later in the present proof.\\	
	{\it Regularity.} We claim that any weak solution $(u,p)$ to \eqref{BVprob} satisfies $(u,p)\in H^2(\Omega_h)\times H^1(\Omega_h)$. This would
be straightforward if $\Omega_h\in  W^{2,\infty} $, see \cite{MurSim74}, but $ R$ is only Lipschitzian. Here, we take advantage of the particular shape of $R$ and use a reflection argument as in \cite{secchi}. We construct a new domain $\Omega^t_h=R^t\setminus B^t_h$,
	obtained by reflecting $\Omega_h$ across $\Gamma_t$, where $R^t=(-L,L)\times [H,3H)$ and $B^t_h$ is the reflection of $B_h$ with respect to $\Gamma_t$. Define $(u^t, p^t) : \Omega^t_h \rightarrow \mathbb{R}^2\times \mathbb{R} $ by
	\begin{equation*}\begin{aligned}
	&u_1^t(x_1,H+x_2)= u_1(x_1, H-x_2), \quad 	&u_2^t(x_1,H+x_2)= -u_2(x_1, H-x_2),\\
	& p^t(x_1, H+x_2)= p(x_1,H-x_2)\quad  &\mbox{for all }(x_1,x_2)\in (-L,L)\times [0,2H)\\[2pt]
	\end{aligned}
	\end{equation*} which satisfies
	\begin{equation}\label{NSpb}
	-\mu\Delta u^t+ u^t\cdot \nabla  u^t + \nabla  p^t =  0,\quad\nabla \cdot u^t =0\quad\mbox{in}\quad \Omega^t_h.\\[2pt]
	\end{equation} Therefore, the couple $$(\overline{u},\overline{p})=\begin{cases}
		(u,p) \ &\mbox{in} \quad \Omega_h,\\[2pt]
		(u^t,p^t) \ &\mbox{in} \quad  \Omega^t_h,
	\end{cases}
$$satisfies the Navier-Stokes equations
	\begin{equation*}
-\mu\Delta \overline{u}+ \overline{u}\cdot \nabla \overline{u} + \nabla  \overline{p} =  0,\quad\nabla \cdot \overline{u} =0\quad\mbox{in}\quad \big\{(-L,L)\times (-H,3H)\big\} \setminus \{B_h \cup B_h^t\} .\\[2pt]
\end{equation*}
	Similarly, let $\Omega^b_h=R^b\setminus B^b_h$ with $R^b=(-L,L)\times (-3H,-H]$ and $B^b_h$ is the reflection of $B_h$ with respect to $\Gamma_b$. Define  $(u^b, p^b) : \Omega^b_h \rightarrow \mathbb{R}^2\times \mathbb{R} $ by
	\begin{equation*}\begin{aligned}
	&u_1^b(x_1,-H-x_2)= u_1(x_1, -H+x_2), \quad \ 	u_2^b(x_1,-H-x_2)= -u_2(x_1, -H+x_2),\\
	& p^b(x_1, -H-x_2)= p(x_1,-H+x_2)   \qquad\ \mbox{for all }(x_1,x_2)\in (-L,L)\times [0,2H)\\[2pt]
	\end{aligned}
	\end{equation*}
	which satisfies the corresponding of \eqref{NSpb} in $\Omega^b_h$. Thanks to these two vertical reflections, we obtain a solution in $\Omega^s_h=\big\{(-L,L)\times (-3H,3H)\big\} \setminus \{B_h \cup B_h^t\cup B_h^b\}$.\\
	With the same principle, we then perform two horizontal reflections of $\Omega^s_h$ with respect to $x_1=\pm L$. At the end of this procedure, let
	$$\widetilde{\Omega}_h=\big\{(-3L,3L)\times (-3H,3H)\big \} \setminus \{B_h \mbox{ and its eight reflections}\}$$ and $(\widetilde{u}, \widetilde{p}): \widetilde{\Omega}_h \rightarrow \mathbb{R}^2\times\mathbb{R}$ be the extension of $(u,p)$, so that
	\begin{equation}\label{NSpb-reflection}
	-\mu\Delta \widetilde{u}+ \widetilde{u}\cdot \nabla  \widetilde{u} + \nabla  \widetilde{p} =  0,\quad
	\nabla \cdot \widetilde{u}=0\quad\mbox{in}\quad \widetilde{\Omega}_h,\qquad
	\widetilde{u}_{|_{\partial B_h}}=0\\[2pt]
	\end{equation}
	and $\tilde{u}$ satisfies further boundary conditions that we do not need to make explicit. After introducing a suitable solenoidal extension, we can proceed as in the first part of the proof and obtain the existence of a solution $(\widetilde{u}, \widetilde{p})\in H^1(\widetilde{\Omega}_h)\times L^2(\widetilde{\Omega}_h)$ satisfying the bounds \eqref{energy-est-gen}-\eqref{L2pressure-est}.
	Hence, $\widetilde{u}\cdot \nabla \widetilde{u} \in L^{3/2}(\widetilde{\Omega}_h)$ and
	\begin{equation}\begin{aligned}\label{estL3/2}
	\|\widetilde{u}\cdot \nabla \widetilde{u} \|_{L^{3/2}(\widetilde{\Omega}_h)}\leq \|\widetilde{u}\|_{L^6(\widetilde{\Omega}_h)} \|\nabla \widetilde{u} \|_{L^2(\widetilde{\Omega}_h)}\leq C\|\widetilde{u}\|^2_{H^1(\widetilde{\Omega}_h)} \leq C_h(\lambda^2 + \lambda^4)\\[2pt]
	\end{aligned}
	\end{equation}
with $C_h=C(\widetilde{\Omega}_h)$. By applying \cite{MurSim74}  and \cite[Theorems IV.4.1 and IV.5.1]{Galdi-steady} to the Stokes problem \eqref{NSpb-reflection}, we infer that
	$(\widetilde{u}, \widetilde{p})\in W^{2,3/2}(\Omega')\times W^{1,3/2}(\Omega')$ for any $ \Omega'\subset \widetilde{\Omega}_h$ and
	\begin{equation}\begin{aligned}\label{estW2-3/2}
	&\|\widetilde{u}\|_{W^{2,3/2}(\Omega')} + \|\widetilde{p}\|_{W^{1,3/2}(\Omega')} \\[2pt]&\leq C_h(\|\widetilde{u}\cdot \nabla \widetilde{u} \|_{L^{3/2}(\widetilde{\Omega}_h)} +	\|\widetilde{u}\|_{W^{1,3/2}(\widetilde{\Omega}_h)}+ \|\widetilde{p}\|_{L^{3/2}(\widetilde{\Omega}_h)})\leq C_h(\lambda + \lambda^4)\\[2pt]
	\end{aligned}
	\end{equation}
with $C_h=C( \Omega', \widetilde{\Omega}_h)$. We recall that $(\widetilde{u}, \widetilde{p})=(u,p)$ in $\Omega_h$. Then, using Sobolev embedding $W^{2,3/2}\hookrightarrow W^{1,6}$ in $\mathbb{R}^2$ and a bootstrap argument we obtain that $(u,p)\in H^2(\Omega_h)\times H^1(\Omega_h)$. Moreover, from \eqref{estL3/2}-\eqref{estW2-3/2} we get
	\begin{equation*}\begin{aligned}
	\|u\|_{H^2(\Omega_h)} + \|p\|_{H^1(\Omega_h)} &\leq C_h(\|\widetilde{u}\cdot \nabla \widetilde{u} \|_{L^{2}( \Omega')} +\|\widetilde{u}\|_{H^1(\Omega')} + \|\widetilde{p}\|_{L^2(\Omega')}  )\\[2pt]&\leq C_h( \|\widetilde{u}\|_{L^3(\Omega')}\|\nabla \widetilde{u}\|_{L^6(\Omega')}  +\|\widetilde{u}\|_{H^1(\Omega')} + \|\widetilde{p}\|_{L^2(\Omega')})\\[2pt]
	&\leq C_h(\|\widetilde{u}\|_{H^1(\Omega')}\|\widetilde{u}\|_{W^{2,3/2}(\Omega')}+\|\widetilde{u}\|_{H^1(\Omega')} + \|\widetilde{p}\|_{L^2(\Omega')})\\[2pt]
	&\leq C_h(\lambda+\lambda^4)\\[2pt]
	\end{aligned}
	\end{equation*}
	with $C_h=C( \Omega_h, \widetilde{\Omega}_h)$. This also proves \eqref{reg-est} whenever $\lambda < \Lambda(h)$.\\[5pt]
	{\it Uniqueness.} Let $u_1$ and $u_2$ be two weak solutions to \eqref{BVprob}, let $w=u_1-u_2$, then
	$$
	\mu\int_{\Omega_h}\nabla w: \nabla \varphi + \int_{\Omega_h}w\cdot \nabla w \cdot \varphi=- \int_{\Omega_h}\left(w\cdot\nabla u_2 + u_2\cdot\nabla w \right)\cdot\varphi
	$$
	for all $\varphi\in W(\Omega_h)$. Then take $\varphi= w$ so that the latter yields
	\begin{equation}\begin{aligned}
	\mu \|\nabla w \|_{L^2(\Omega_h)}^2\! &= - \int_{\Omega_h}\!\! w\cdot \nabla u_2 \cdot w
	\leq \| \nabla u_2\|_{L^2(\Omega_h)}\| w\|_{L^4(\Omega_h)}^2 \\[2pt]&\leq C_h(1+\tfrac{1}{\mu}) (\lambda+ \lambda^2)\| \nabla w\|_{L^2(\Omega_h)}^2,\\[2pt]
	\end{aligned}
	\end{equation}
	where we used H\"older, Ladyzhenskaya and Poincaré inequalities and \eqref{energy-est-gen}.
	Hence, there exists $\Lambda=\Lambda(h)>0$ (uniformly upper-bounded with respect to $h$) such that
	\begin{equation}\label{smallness}
	\lambda\in[0,\Lambda(h))\ \Longleftrightarrow\
	C_h(1+\tfrac{1}{\mu})(\lambda+\lambda^2)<\mu\\[2pt]
	\end{equation}
	and this condition implies $\|\nabla w\|_{L^2(\Omega_h)}=0$ and, in turn, $w=0$ since $w_{|_{\partial \Omega_h}}=0$.\\[5pt]
	{\it Refined bounds.} For $\lambda\in[0,\Lambda(h))$, in all the above bounds we can drop the largest power of $\lambda$ and they all become linear upper bounds. We treat separately the cases $U=1$ and $U=0$ and we make explicit the dependence of the constant $C_h$ in \eqref{energy-est-gen} on $h$.\\
	When $U=1$, we claim that the unique strong solution $u$ to \eqref{BVprob} satisfies
	\begin{equation}\label{energy-est-eps}
	\| u\|_{H^1(\Omega_h)}  \leq C\big(1+(\eps_t(h))^{-3/2}\big)\lambda\\[2pt]
	\end{equation}
	with $C>0$ independent of $h$.
	To this end, we introduce a different (and explicit) solenoidal extension.
	Consider the cut-off functions $\zeta_l,\zeta_r\in C^\infty(\mathbb{R}^2)$, with $0\leq \zeta_{ l}, \zeta_r\leq 1$, defined piece-wise in the rectangles of Figure \ref{cutoffsU=1} by
	\begin{equation}\label{zeta-left}
	\zeta_l(x_1,x_2)=\begin{cases}
	0 \quad&\mbox{in} \quad  [-\tau,\tau]\times [-H, H -\tfrac{\eps_t(h) }{2}],  \\[2pt]0 \quad&\mbox{in} \quad [\tau, L]\times[-H,H] ,\\[2pt]
	1 \quad&\mbox{in} \quad  [-L,-2\tau]\times [-H,H],\\[2pt]
	\zeta^0_l(x_1)\quad&\mbox{in}\quad [-2\tau,-\tau]\times [-H,H-\tfrac{\eps_t(h)}{2}],\\[2pt]
	C^\infty\mbox{-completion}\quad&\mbox{in} \quad [-2\tau,-\tau]\times [H-\tfrac{\eps_t(h)}{2},H],
	\end{cases}
	\end{equation}where $\zeta^0_l$ is a function only of $x_1$, and
	\begin{equation}\label{zeta-right}
	\zeta_r(x_1,x_2)=\begin{cases}
	0 \quad&\mbox{in} \quad [-\tau,\tau]\times [-H, H -\tfrac{\eps_t(h)}{2}],  \\[2pt]0 \quad&\mbox{in} \quad [-L, -\tau]\times[-H,H-\tfrac{ \eps_t(h)}{4}] ,\\[2pt]
	1\quad&\mbox{in} \quad[2\tau, L]\times[-H,H-\tfrac{ \eps_t(h)}{4}],\\[2pt]
	\zeta^0_r(x_1)\quad&\mbox{in}\quad [\tau,2\tau]\times [-H,H-\tfrac{ \eps_t(h)}{2}],\\[2pt]
		1-\zeta_l(x_1,x_2)& \mbox{in}  \quad  [-L,L]\times [H -\tfrac{\eps_t(h)}{4},H ],\\[2pt]
	C^\infty\mbox{-completion}	\quad&\mbox{in}\quad [-\tau,2\tau]\times [H-\tfrac{ \eps_t(h)}{2}, H-\tfrac{ \eps_t(h)}{4}],\\[2pt]
	\end{cases}
	\end{equation}
	where $\zeta^0_r$ is a function only of $x_1$.
	\begin{figure}[!h]
		\begin{center}
			\includegraphics[scale=0.13]{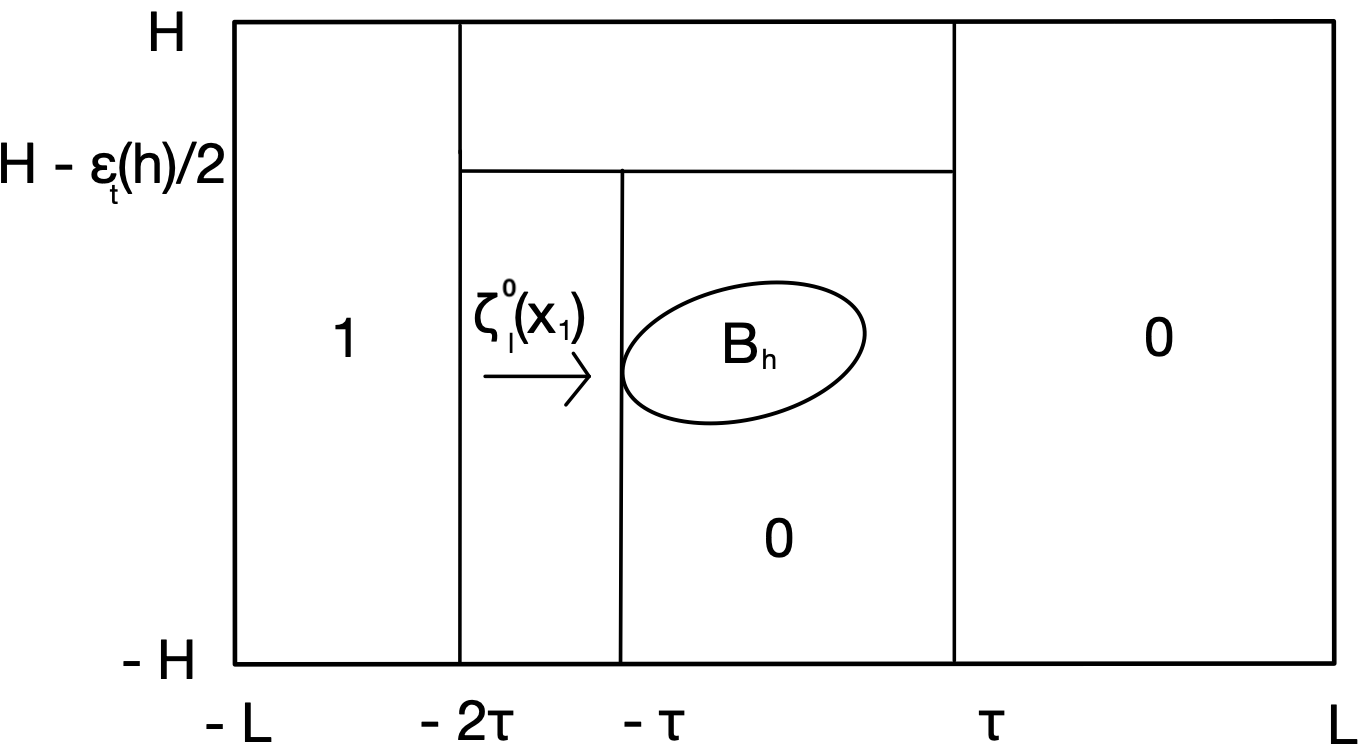}\includegraphics[scale=0.13]{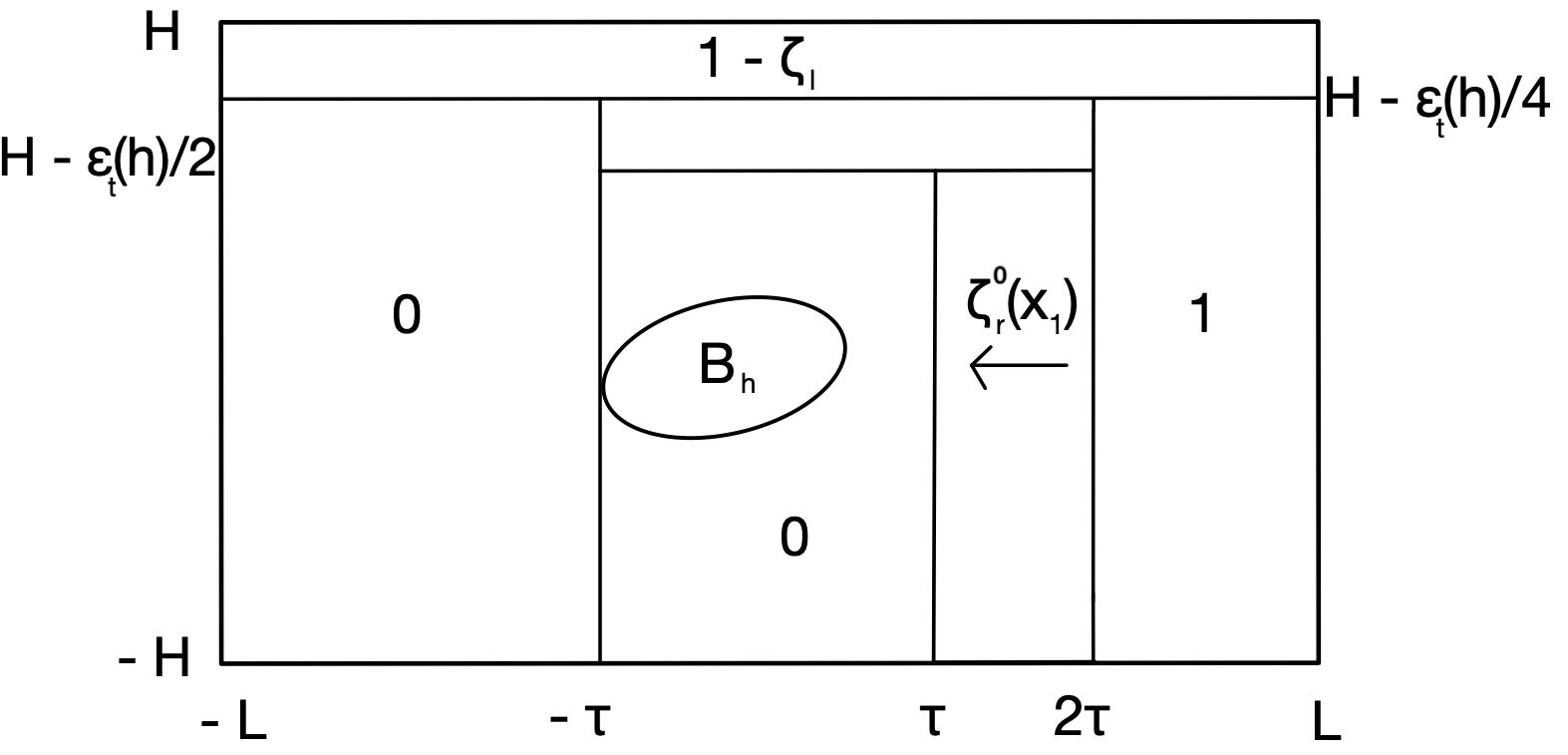}
			\caption{The cut-off functions $\zeta_l$ (left) and $\zeta_r$ (right) on $\overline{ R}$ when $U=1$.}
			\label{cutoffsU=1}
		\end{center}
	\end{figure}

	Then, letting $\nabla^\perp=(-\partial_2,\partial_1)$, consider the vector field $s: R\rightarrow \mathbb{R}^2$ defined by	
	\begin{equation}\label{solext}\begin{aligned}
	s(x_1,x_2):=-\lambda\nabla^\perp\left(\zeta_l(x_1,x_2)\int_{-H}^{x_2} V_{\rm in}(z)dz +  \zeta_r(x_1,x_2)\int_{-H}^{x_2} V_{\rm out}(z)dz\right),\\[2pt]
	\end{aligned}
	\end{equation}
	which is solenoidal and satisfies the boundary conditions in \eqref{BVprob}. Rewriting $s$ as
	\begin{equation*}
	\begin{aligned}
	s(x_1,x_2)=  \lambda\left(-\nabla^\perp \zeta_l \int_{-H}^{x_2}V_{\rm in} - \nabla^\perp \zeta_r  \int_{-H}^{x_2}V_{\rm out}+ (\zeta_l V_{\rm in} + \zeta_r V_{\rm out}) e_1\right),\\[2pt]
	\end{aligned}
	\end{equation*}
	its partial derivatives read
	\begin{equation*}
	\begin{aligned}
	\partial_1 s= \lambda \left(-\nabla^\perp \partial_1 \zeta_l \int_{-H}^{x_2}V_{\rm in} - \nabla^\perp \partial_1\zeta_r \int_{-H}^{x_2}V_{\rm out} + (\partial_1 \zeta_l V_{\rm in} + \partial_1 \zeta_r V_{\rm out})e_1\right)\, ,\\[2pt]
	\end{aligned}
	\end{equation*}
	\begin{equation*}
	\begin{aligned}
	\partial_2 s= \lambda \bigg( & -\nabla^\perp \partial_2 \zeta_l \int_{-H}^{x_2}V_{\rm in} - \nabla^\perp \partial_2\zeta_r \int_{-H}^{x_2}V_{\rm out} - \nabla^\perp\zeta_l V_{\rm in} -\nabla^\perp\zeta_r V_{\rm out} \\[2pt]& + (\partial_2 \zeta_l V_{\rm in} + \partial_2 \zeta_r V_{\rm out} +\zeta_l \tfrac{d}{dx_2}V_{\rm in} +  \zeta_r \tfrac{d}{dx_2}V_{\rm out})e_1\bigg).\\[2pt]
	\end{aligned}
	\end{equation*}
	Using that $V_{\rm in},V_{\rm out}\in W^{2,\infty}(-H,H)$ and that $\zeta_l,\zeta_r$ are smooth, it follows that
	\begin{equation}\label{implicit-solext-est}
	\begin{aligned}
	\|s\|_{L^\infty(\Omega_h)}, \|s\|_{L^2(\Omega_h)}, \ &\|s\|_{L^4(\Omega_h)} , \ \|\nabla s\|_{L^2(\Omega_h)}, \
	\|\Delta s\|_{L^2(\Omega_h)}
\leq  C_h\lambda,\\[2pt]
	&\|s\cdot \nabla s\|_{L^2(\Omega_h)}\leq  C_h\lambda^2\le C_h\lambda.\\[2pt]
	\end{aligned}
	\end{equation}
	 We need to quantify the dependence of $C_h>0$ on $\eps_b(h)$ and $\eps_t(h)$. On the one hand, we notice that, by construction, both $\zeta_l$ and $\zeta_r$ depend on $x_2$ only in \begin{equation}\label{Omegaeps1}\Omega_{\eps_t(h)}:=[-2\tau, 2\tau] \times [H-\tfrac{\eps_t(h)}{2},H].\\[2pt]\end{equation}
	In this domain the $x_1$-derivatives of $\zeta_{l}$ and $\zeta_r$ are uniformly bounded with respect to $h$ while the $x_2$-derivatives blow-up as $\eps_t(h)$ goes to zero, for instance we have
	$$|\partial_2 \zeta_{l}|, |\partial_2 \zeta_{r}|\leq C (\eps_t(h))^{-1},\qquad |\partial^2_2 \zeta_{r}|, |\partial^2_2 \zeta_{l}|\leq C (\eps_t(h))^{-2}. $$
	Therefore, in $\Omega_{\eps_t(h)}$
	\begin{equation*}
	\begin{aligned}
	|s|\leq C (1+&(\eps_t(h))^{-1})\lambda, \quad |\partial_1 s|\leq C (1+(\eps_t(h))^{-1})\lambda,\\& |\partial_2 s|\leq C ( (\eps_t(h))^{-1}+(\eps_t(h))^{-2})\lambda.
	\end{aligned}
	\end{equation*}
	On the other hand, the cut-off functions depend only on $x_1$  in $\Omega_h\setminus \Omega_{\eps_t(h)}$ and their $x_1$ and $x_2$-derivatives are uniformly bounded with respect to $h$. Therefore, in $\Omega_h\setminus \Omega_{\eps_t(h)}$
	\begin{equation*}
	|s|, |\partial_1 s|, |\partial_2 s|\leq C \lambda.
	\end{equation*}
	Gathering all together, we refine the bounds in \eqref{implicit-solext-est} as
	\begin{equation}\begin{aligned}\label{solext-eps1}
&	\| s\|_{L^\infty(\Omega_h)}\leq C (1+(\eps_t(h))^{-1})\lambda,\\[2pt]
&	\|s\|_{L^2(\Omega_h)}
	\leq C\lambda +
	C \left(\int_{\Omega_{\eps_t(h)}}( \eps_t(h))^{-2}\right)^{1/2}\lambda \leq  C (1+ (\eps_t(h))^{-1/2})\lambda,\\
	&\|s\|_{L^4(\Omega_h)}\leq C (1+(\eps_t(h))^{-3/4})\lambda,\qquad \|\nabla s\|_{L^2(\Omega_h)} \leq C (1+(\eps_t(h))^{-3/2})\lambda,\\[2pt]
		&\|\Delta s\|_{L^2(\Omega_h)}, \ \| s\cdot \nabla s\|_{L^2(\Omega_h)}\leq C (1+(\eps_t(h))^{-5/2})\lambda,\\[2pt]
	\end{aligned}
	\end{equation}with all the constants $C>0$ independent of $h$.
	Then, testing \eqref{homo-prob} with $v=u-s$ we obtain
	\begin{equation}\label{ident-prov2}
	\begin{aligned}
	\mu\|\nabla v\|_{L^2(\Omega_h)}^2= - \int_{\Omega_h} v\cdot \nabla s\cdot v -\int_{\Omega_h}s\cdot \nabla s \cdot v  - \mu \int_{\Omega_h}\nabla s : \nabla v\\[2pt]
	\end{aligned}
	\end{equation}
	We want to estimate, when possible, only $s$ and not $\nabla s$ since the bounds for $s$ are less singular in terms of $\eps_t(h)$. Hence, since $\nabla\cdot v=\nabla\cdot s=0$ and using integration by parts, we rewrite \eqref{ident-prov2} as
	\begin{equation}\label{ident-prov3}
	\begin{aligned}
	\mu\|\nabla v\|_{L^2(\Omega_h)}^2=& \int_{\Omega_h} v\cdot \nabla v\cdot s +\int_{\Omega_h}s\cdot \nabla v \cdot s  - \mu \int_{\Omega_h}\nabla s : \nabla v.\\[2pt]
	\end{aligned}
	\end{equation}
	We split the first integral in the right-hand side over $\Omega_{\eps_t(h)}$ and $\Omega_h \setminus\Omega_{\eps_t(h)}$. On the one hand, since $v_{|_{\Gamma_t}}=0$, Poincaré inequality
	\begin{equation*}
	\|v\|_{L^2{(\Omega_{\eps_t(h)}})}\leq \frac{\eps_t(h) }{2}\|\nabla v\|_{L^2{(\Omega_{\eps_t(h)}})},\\[2pt]
	\end{equation*}
and H\"older inequality yield
	\begin{equation*}\begin{aligned}
	\int_{\Omega_{\eps_t(h)}}( v \cdot \nabla v )\cdot s &\leq \|v\|_{L^2{(\Omega_{\eps_t(h)}})}	\|\nabla v\|_{L^2{(\Omega_{\eps_t(h)}})}	\|s\|_{L^\infty{(\Omega_{\eps_t(h)}})}\\
	&\leq C\eps_t(h)  \|\nabla v\|^2_{L^2{(\Omega_{\eps_t(h)}})}  (1+ (\eps_t(h))^{-1})\lambda\leq C \lambda   \|\nabla v\|^2_{L^2{(\Omega_{\eps_t(h)}})},\\[2pt]
	\end{aligned}
	\end{equation*}where we used that $\|s\|_{L^\infty(\Omega_{\eps_t(h)})}\leq C(1+ (\eps_t(h))^{-1})\lambda$ and $\eps_t(h)\leq 2H-\delta_b -\delta_t$. On the other hand, since $v_{|_{\Gamma_l, \Gamma_r}}=0$, Poincaré and H\"older inequalities yield
	\begin{equation*}\begin{aligned}
	\int_{\Omega_h \setminus\Omega_{\eps_t(h)}}( v \cdot \nabla v )\cdot s &\leq  \|v\|_{L^2(\Omega_h\setminus\Omega_{\eps_t(h)})}	\|\nabla v\|_{L^2{(\Omega_h \setminus\Omega_{\eps_t(h)})}}	\|s\|_{L^\infty(\Omega_h \setminus\Omega_{\eps_t(h)})}\\& \leq C\lambda \|\nabla v\|^2_{L^2{(\Omega_h \setminus\Omega_{\eps_t(h)})}},\\[2pt]
	\end{aligned}
	\end{equation*}where we used that $\|s\|_{L^\infty(\Omega_h \setminus\Omega_{\eps_t(h)})}\leq C\lambda$.
	Therefore, from \eqref{solext-eps1} and \eqref{ident-prov3} we infer
	\begin{equation*}\begin{aligned}
	\mu\|\nabla v\|^2_{L^2(\Omega_h)}&\leq C\lambda\|\nabla v\|^2_{L^2(\Omega_h)} + \|s\|_{L^4(\Omega_h)}^2 \|\nabla v\|_{L^2(\Omega_h)}  +\mu \|\nabla s\|_{L^2(\Omega_h)}\|\nabla v\|_{L^2(\Omega_h)}\\[2pt]
	&\leq C\lambda \|\nabla v\|^2_{L^2(\Omega_h)}+ C (1+( \eps_t(h))^{-3/2})(\lambda+\lambda^2)\|\nabla v\|_{L^2(\Omega_h)}.\\[2pt]
	\end{aligned}
	\end{equation*}
	Then, for $\lambda\in[0,\Lambda(h))$ with $\Lambda(h)$ as in \eqref{smallness} we have
	\begin{equation}\label{energy-est-esp-v}
	\begin{aligned}
	\|\nabla v\|_{L^2(\Omega_h)} \leq C(1+ (\eps_t(h))^{-3/2})\lambda\\[2pt]
	\end{aligned}
	\end{equation}and
	\begin{equation*}\begin{aligned}
	\|u\|_{H^1(\Omega_h)}\leq \|\nabla v\|_{L^2(\Omega_h)} + \|s\|_{H^1(\Omega_h)}\leq C(1+(\eps_t(h))^{-3/2})\lambda ,\\[2pt]
	\end{aligned}
	\end{equation*}which proves \eqref{energy-est-eps}.\\
	When $U=0$, we claim that the unique strong solution $u$ to \eqref{BVprob} satisfies
	\begin{equation}\label{energy-est-bound}
	\| u\|_{H^1(\Omega_h)}  \leq C\lambda\\[2pt]
	\end{equation}
	with $C>0$ independent of $h$, which will imply that $\Lambda(h)\equiv \Lambda$ can be also taken independent of $h$.
	In this case, we shall define the cut-off functions and the solenoidal extension differently depending  if $h\leq 0$ or $h>0$. If $h\leq 0$, we define $\zeta_l$, $\zeta_r$ as in \eqref{zeta-left}-\eqref{zeta-right}  (see Figure \ref{cutoffsU=0} below) replacing $\eps_t(h)$ with the distance of $B_0$ to $\Gamma_t$, namely $\eps_t(0)=H-\delta_t$. The solenoidal extension $s$ is then defined as in \eqref{solext}.
	By construction both $\zeta_l$ and $\zeta_r$ depend on $x_2$ only in $\Omega_{\eps_t(0)}$, defined as in \eqref{Omegaeps1} with $\eps_t(h)$ replaced by $\eps_t(0)$. In this domain both $x_1$ and $x_2$-derivatives of $\zeta_l$ and $\zeta_r$ are uniformly bounded with respect to $h$, for instance we have
	$$|\partial_2 \zeta_{l}|, |\partial_2 \zeta_{r}|\leq C( \eps_t(0))^{-1}\leq C,\qquad |\partial^2_2 \zeta_{r}|, |\partial^2_2 \zeta_{l}|\leq C (\eps_t(0))^{-2}\leq C. $$
	Since in $\Omega_h \setminus \Omega_{\eps_t(0)}$ the cut-off functions depend only on $x_1$,
	we infer that $s$, $\partial_1 s$ and $\partial_2 s$ are uniformly bounded with respect to $h$ in all $\Omega_h$ and
	\begin{equation}\label{s-bounded}
	\|s\|_{L^\infty(\Omega_h)}, 	\|s\|_{L^2(\Omega_h)}, 	\|s\|_{L^4(\Omega_h)},	\|\nabla s\|_{L^2(\Omega_h)} \leq C\lambda.\\[2pt]
	\end{equation}Repeating the same computations as in the case $U=1$ and using \eqref{s-bounded}, we obtain \eqref{energy-est-bound} for $h\leq0$.\\
		If $h>0$, we make a vertical reflection $x_2\mapsto -x_2$ and we consider the new cut-off functions defined piece-wise in the rectangles of Figure \ref{cutoffsU=0}, where $\eps_b(0)=H-\delta_b$.
	\begin{figure}[!h]
		\begin{center}
			\includegraphics[scale=0.13]{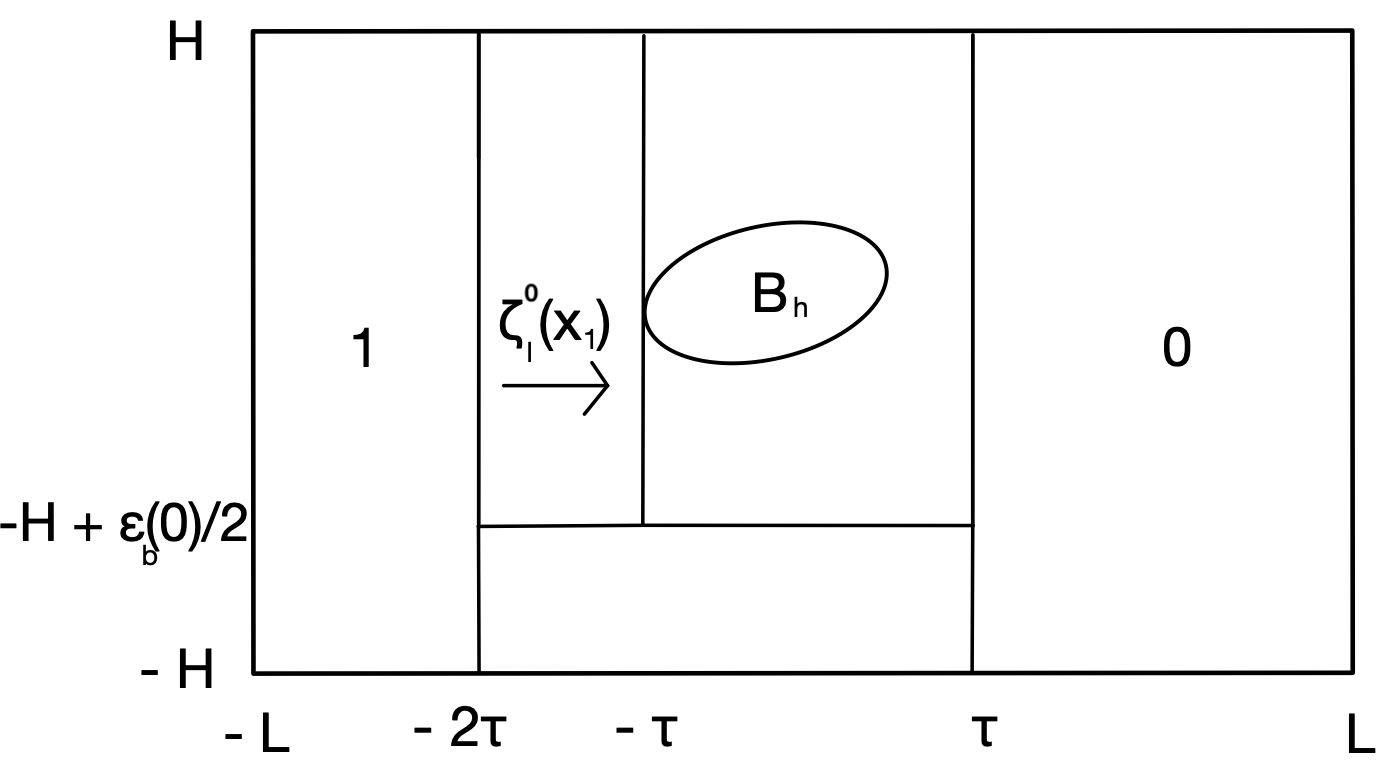}\includegraphics[scale=0.13]{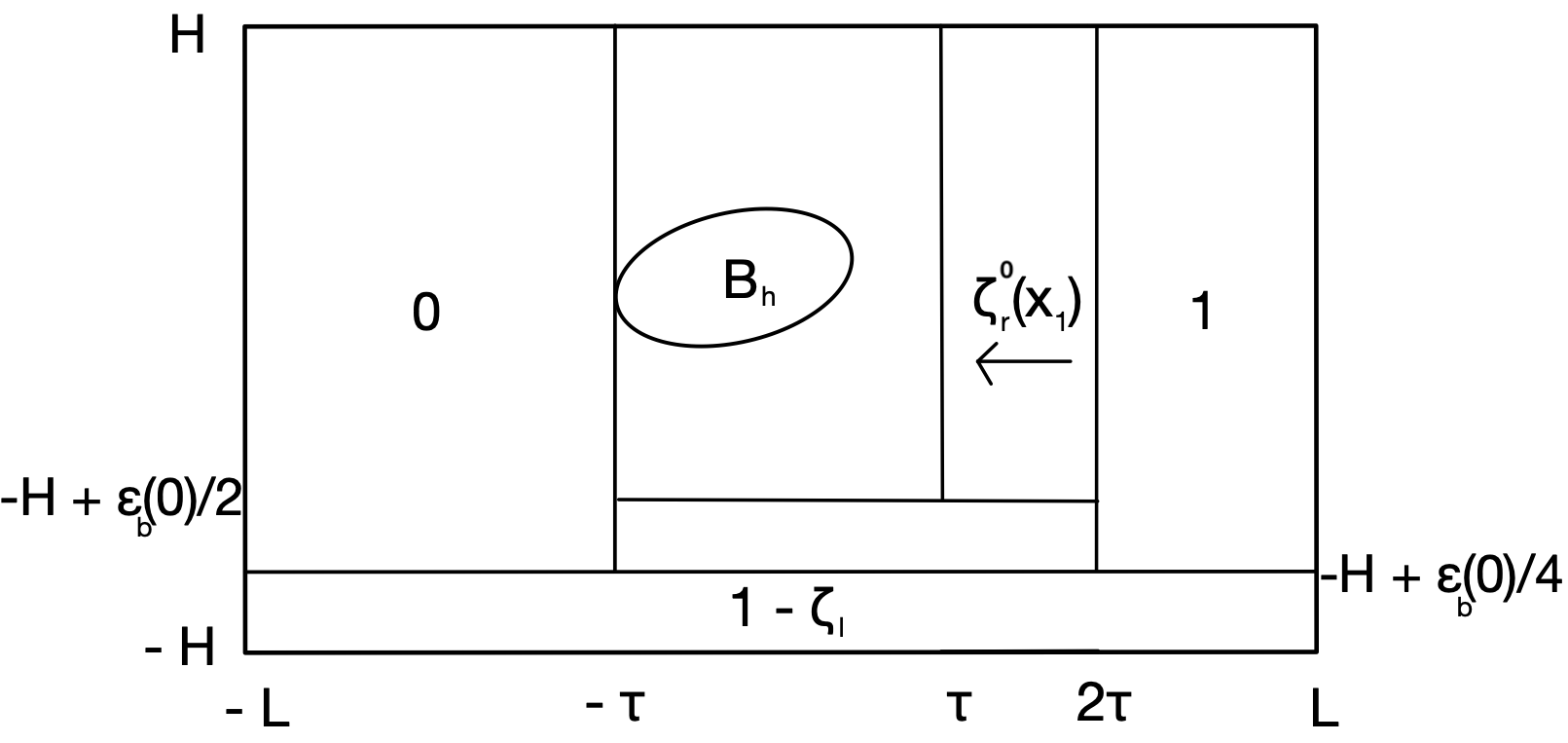}
			\caption{The cut-off functions $\zeta_l$ (left) and $\zeta_r$ (right) on $\overline{ R}$ when $U=0$ for $h>0$.}
			\label{cutoffsU=0}
		\end{center}
	\end{figure}
	
Then, we consider the vector field $s: R\rightarrow \mathbb{R}^2$ defined by	
	\begin{equation*}\begin{aligned}
	s(x_1,x_2):=\lambda\nabla^\perp\left(\zeta_l(x_1,x_2)\int_{x_2}^{H} V_{\rm in}(z)dz +  \zeta_r(x_1,x_2)\int_{x_2}^{H} V_{\rm out}(z)dz\right),\\[2pt]
	\end{aligned}
	\end{equation*}
	which is solenoidal and satisfies the boundary conditions in \eqref{BVprob}. By the same argument used when $h\leq0$, $s$, $\partial_1 s$ and $\partial_2 s$  are uniformly bounded with respect to $h$ in $\Omega_h$. Therefore, using again \eqref{s-bounded}, we obtain \eqref{energy-est-bound} for $h<0$.\end{proof}

\begin{remark}We stated \eqref{energy-est} and \eqref{reg-est} only in case of uniqueness because, in what follows, $\lambda$ will be taken small and
higher powers of $\lambda$ can be upper estimated with the first power.\par
The reflection method used to obtain the regularity result has its own interest. The rectangular shape of the domain is crucial and the technique fails for other polygons. However, in the case of convex polygons, in particular also for a rectangle, one can obtain the more $C^\infty$-regularity result by using Theorem 2 in \cite{KelOsb76}, see also \cite[Section 7.3.3]{Gris85} and \cite{Dauge82}.
\end{remark}

\section{Equilibrium configurations of a FSI problem}\label{sec-FSI}
By Theorem \ref{strong-exiuni}, for any 
$(\lambda, h)\in [0,+\infty)\times (-H+\delta_b, H-\delta_t)$ there exists at least a strong solution $(u,p)=(u(\lambda, h), p(\lambda, h))$ to \eqref{BVprob}. The fluid described by $(u,p)$ in $\Omega_h$ exerts on $B_h$ a force perpendicular to the direction of the inflow, called \textit{lift} (see \cite{PaPriDeLan10}). Since the inflow in \eqref{BVprob} is horizontal, the lift is vertical and given by
\begin{equation}\label{lift}
\mathcal{L}(\lambda, h)=-e_2\cdot \int_{\partial B_h} \mathbb{T}(u,p)n,\\[2pt]
\end{equation}
where $\mathbb{T}$ is the fluid stress tensor, namely
\begin{equation*}
\mathbb{T}(u,p):= \mu(\nabla u + \nabla u^T) -p\mathbb{I}\, ,\\[2pt]
\end{equation*}
and $n$ is the unit outward normal vector to $\partial \Omega_h$, which, on $\partial B_h$, points towards the interior of $B_h$. In fact, $\mathcal{L}(\lambda, h)$ is a multi-valued function when uniqueness for \eqref{BVprob} fails. However, we keep this simple notation instead of writing $\mathcal{L}(\lambda, h, u(\lambda, h), p(\lambda, h)),$ in which also the dependence on the particular solution $(u,p)$ is emphasized. The regularity of the solution (see Theorem \ref{strong-exiuni}) and the smoothness of $\partial B_h$ yield $\mathbb{T}(u,p)_{|_{\partial B_h}}\in H^{1/2}(\partial B_h)\subset L^1(\partial B_h)$, hence the integral in \eqref{lift} is finite. In fact, the lift can also be defined for merely weak solutions, see \eqref{weaklift} in Section \ref{suspension}. Note that \eqref{lift} holds for any $\lambda\geq 0$ and any solution to \eqref{BVprob} but our main result on the FSI problem focuses on small inflows, see Theorem \ref{main-theo}.
\par

Aiming to model, in particular, a wind flow hitting a suspension bridge, the body $B$ may also be subject to a (possibly nonsmooth)
vertical restoring force $f$ tending to maintain $B$ in the equilibrium position $B_0$ (for $h=0$); see Section \ref{suspension}. We assume that $f$ depends only on the position $h$, that $f\in C^0(-H+\delta_b,H-\delta_t)$ with $f(0)=0$ and
\begin{equation}\label{f}
\exists\gamma>0\quad\mbox{s.t.}\quad\frac{f(h_1)-f(h_2)}{h_1-h_2}\geq\gamma\quad\forall h_1,h_2 \in (-H+\delta_b,H-\delta_t), \ h_1\neq h_2.\\[2pt]
\end{equation}
Moreover, we assume that there exists $K>0$ such that
\begin{equation}\label{f-prop}
\begin{aligned}
&\limsup\limits_{h\rightarrow -H+\delta_b} \ f(h)(H-\delta_b+h)^{3/2}\leq -K,
\\[2pt]
&\liminf\limits_{h\rightarrow H-\delta_t} \ \frac{f(h)}{\max \{(H-\delta_t-h)^{-3/2},U(H-\delta_t-h)^{-3}\}}\geq K.\\[2pt]
\end{aligned}
\end{equation}
The assumption \eqref{f-prop} is somehow technical and prevents collisions of $B$ with the horizontal boundary
$\Gamma_b\cup\Gamma_t$, at least for small inflow/outflow. It can probably be relaxed but, so far, only few (numerical) investigations on the effect of proximity to collisions of hydrodynamic forces (such as the lift), acting on non-spherical bodies, have
been tackled, see \cite{Zarghami} and references therein.
The presence of $U$ in \eqref{f-prop} highlights the different  behavior of $f$ when $B$ is close to $\Gamma_t$
for $U=0$ or $U=1$. In the first case,  $f$ has the same strength close to $\Gamma_b$ and $\Gamma_t$. Conversely, for $U=1$, the asymmetry of the boundary conditions requires a different strength of $f$, which is stronger when $B$ is close to $\Gamma_t$ than when $B$ is close to $\Gamma_b$.
Overall, \eqref{f}-\eqref{f-prop} model the fact that $B$ is not allowed to go too
far away from the equilibrium position $B_0$.\\
Since we are interested in the equilibrium configurations of the FSI problem, we consider the boundary-value problem \eqref{BVprob} coupled with a compatibility condition stating that the restoring force balances the lift force, namely
\begin{equation}\label{FSI}
\begin{array}{cc}
-\mu\Delta u + u\cdot \nabla u  + \nabla p = 0,\quad\nabla \cdot u =0\quad\mbox{in}\quad \Omega_h\\[2pt]
u_{|_{\partial B_h}}\!=\!u_{|_{\Gamma_b}}\!=0, \quad u_{|_{\Gamma_t}}\!=\lambda Ue_1,\quad
u_{|_{\Gamma_l}}\!=\lambda V_{\rm in}(x_2)e_1,\quad u_{|_{\Gamma_r}}\!=\lambda V_{\rm out}(x_2)e_1, \\[5pt]
f(h)=-e_2\cdot \int_{\partial B_h} \mathbb{T}(u,p)n.\\[2pt]
\end{array}
\end{equation}
Our main result concerns the existence and uniqueness of the solution to \eqref{FSI} for small values of $\lambda$, that we expect to be stable.

\begin{theorem}\label{main-theo}
	Let $f\in C^0(-H+\delta_b, H-\delta_t)$ satisfy \eqref{f}-\eqref{f-prop} with $f(0)=0$ and $V_{\rm in}$, $V_{\rm out}\in W^{2,\infty}(-H,H)$ satisfy \eqref{match-conds} with \eqref{U01}. There exists $\Lambda_1>0$ and a unique $\mathfrak{h}\in C^0[0, \Lambda_1)$ such that for $\lambda\in[0,\Lambda_1)$ the FSI problem \eqref{FSI} admits a unique solution
	$(u(\lambda,h), p(\lambda,h),h)\in H^2(\Omega_h)\times H^1(\Omega_h)\times (-H+\delta_b, H-\delta_t)$
	given by
	\begin{equation*}(u(\lambda,\mathfrak{h}(\lambda)), p(\lambda,\mathfrak{h}(\lambda)),\mathfrak{h}(\lambda)).\\[2pt]
	\end{equation*}
\end{theorem}
We emphasize that Theorem \ref{main-theo} ensures uniqueness of the equilibrium configuration for the FSI problem  \eqref{FSI} in the \textit{uniform} interval $[0, \Lambda_1)$ even in absence of uniqueness for \eqref{BVprob} that, instead,  is only ensured in the possibly \textit{non-uniform} interval $[0,\Lambda(h))$.
The proof of Theorem \ref{main-theo} is given in Section \ref{sec-proofMain}. It is fairly delicate because if $U=0$ (as for symmetric inflow/outflow), then from \eqref{energy-est-eps} we infer that the $H^1$-norm is uniformly bounded with respect to $h$.
However, if $U=1$, the same norm obviously blows up when $B_h$ approaches $\Gamma_t$, which affects the bounds
for the lift in \eqref{lift}. As already mentioned, very little is known when a body approaches a collision,
see again \cite{Zarghami} and references therein. Therefore, the next statement has its own independent interest, it provides some upper
bounds and shows that, probably, the lift behaves differently for homogeneous and inhomogeneous boundary data.
\begin{theorem}\label{lift-theo}
Assume \eqref{U01} and let $\lambda\in [0,\Lambda_0]$ for some $\Lambda_0>0$. Let $(u,p)$ be a strong
	solution to \eqref{BVprob} (see Theorem \ref{strong-exiuni}) and let $\mathcal{L}(\lambda, h)$ be as in \eqref{lift}. There exists $C>0$ (independent of $\lambda, h, u, p$) such that, for any  $(\lambda,h )\in [0,\Lambda_0]\times(-H+\delta_b, H-\delta_t)$,
	\begin{equation}\label{est-lift}
	|\mathcal{L}(\lambda, h)|\leq C \Big((\eps_b(h))^{-3/2}+ \max\{(\eps_t(h))^{-3/2},U (\eps_t(h))^{-3} \}\Big)\lambda\\[2pt]
\end{equation}
	with $\eps_b(h)$ and $\eps_t(h)$ defined in \eqref{parameters}. In fact, $\mathcal{L}(\lambda, h)$ is defined in all $[0,+\infty)\times(-H+\delta_b, H-\delta_t)$,  possibly as a multi-valued function,  but \eqref{est-lift} would hold with different powers of $\lambda$.
\end{theorem}

The proof of Theorem \ref{lift-theo} is given in the next section.

\section{Proof of Theorem \ref{lift-theo}}\label{sec-proofTheos}

We rewrite the lift \eqref{lift}, which is a boundary integral, as a volume integral. This can be done by considering $w\in H^1(\Omega_h)$ that satisfies
	\begin{equation}\label{dual-bogo}
	\nabla \cdot w=0\quad \mbox{in}\quad \Omega_h,\qquad w_{|_{\partial B_h}}=e_2,\qquad w_{|_{\partial R}}=0.\\[2pt]
	\end{equation}
	The Divergence Theorem ensures that \eqref{dual-bogo} admits infinitely many solutions.
	Testing \eqref{BVprob} with one such solution $w$ (recall that $\nabla \cdot \mathbb{T}=\mu \Delta u -\nabla p$) yields
	\begin{equation*}
	\int_{\Omega_h} u\cdot \nabla u \cdot w = \int_{\Omega_h}\nabla\cdot \mathbb{T}(u,p)\cdot w=-\mu\int_{\Omega_h}\nabla u : \nabla w + \int_{\partial\Omega_h}\mathbb{T}(u,p)n \cdot w\\[2pt]
	\end{equation*}
	and, using the boundary conditions on $w$,
	\begin{equation}\label{lift-volint}
	-e_2\cdot \int_{\partial B_h}\mathbb{T}(u,p)n = -\int_{\Omega_h} u\cdot \nabla u \cdot w - \mu\int_{\Omega_h}\nabla u : \nabla w.\\[2pt]
	\end{equation}
	Among the infinitely many solutions of \eqref{dual-bogo}, we select one obtained by using a solenoidal extension similar to the ones introduced in Section \ref{fluid-section}. We consider a cut-off function $\chi\in C^\infty(\overline{R})$ with $0\leq\chi\leq 1$ such that
	\begin{equation*}\begin{aligned}
	\chi(x_1,x_2)=\begin{cases}
	1 \quad &\mbox{in} \quad [-\tau, \tau]\times [h-\delta_b, h+\delta_t],\\[2pt]
	0 \quad &\mbox{in} \quad\Omega_h \setminus ( [-2\tau, 2\tau]\times [h-\delta_b- \tfrac{\eps_b(h)}{2}, h+\delta_t + \tfrac{\eps_t(h)}{2}]),\\[2pt]
	\chi^0(x_1) \quad &\mbox{in}\quad  ([-2\tau,-\tau]\cup [\tau,2\tau])\times [h-\delta_b, h+\delta_t], \\[2pt]
	C^\infty\mbox{-completion} \quad &\mbox{elsewhere}.
	\end{cases}
	\end{aligned}
	\end{equation*}
	We put $w=\nabla^\perp(x_1 \chi)$. Clearly $w\in H^1(\Omega_h)$ satisfies \eqref{dual-bogo} and $\mathrm{supp} \ w \subseteq \Omega_w=   \Omega_{w,b} \cup  \Omega_{w,c} \cup  \Omega_{w,t}$ with
	\begin{equation*}
	\begin{aligned}
	\Omega_{w,b}&:=[-2\tau,2\tau]\times [h-\delta_b- \tfrac{\eps_b(h)}{2}, h-\delta_b],\qquad\Omega_{w,c}:=[-2\tau,2\tau]\times [h-\delta_b, h+\delta_t],\\[2pt]
	\Omega_{w,t}&:=[-2\tau,2\tau]\times [h+\delta_t, h+\delta_t + \tfrac{\eps_t(h)}{2}].\\[2pt]
	\end{aligned}
	\end{equation*}
	Moreover, from the definition  of $\chi$ it follows that $w$ and its $x_1$ and $x_2$-derivatives are uniformly bounded with respect to $h$ in $\Omega_{w,c}$, while in $\Omega_{w,b}$
	\begin{equation}\label{w-bounds-bot}\begin{aligned}
	|w|\leq C  (1&+(\eps_b(h))^{-1}),  \quad 	|\partial_1 w|\leq (1+(\eps_b(h))^{-1}),\\& |\partial_2 w|\leq ((\eps_b(h))^{-1}+(\eps_b(h))^{-2})\\[2pt]
	\end{aligned}
	\end{equation} and in $ \Omega_{w,t}$
	\begin{equation}\label{w-bounds-top}\begin{aligned}
	|w|\leq C  (1&+(\eps_t(h))^{-1}),  \quad 	|\partial_1 w|\leq (1+(\eps_t(h))^{-1}),\\& |\partial_2 w|\leq ((\eps_t(h))^{-1}+(\eps_t(h))^{-2}).\\[5pt]
		\end{aligned}
	\end{equation}
	{\it $B_h$ close  to $\Gamma_b$.} We consider the case when $h$ is close to $-H+\delta_b$, hence $\eps_b(h)$ is close to zero. This implies that $\eps_t(h)\geq 1$ and the bounds in \eqref{w-bounds-top} become uniform. Choosing in \eqref{lift-volint} the previously constructed $w$, we observe that the integrals in the right-hand side are defined only on $\Omega_w$. Let us split these integrals over the regions $\Omega_{w,b}$, which is shrinking as $\eps_b(h)$ goes to zero, and $\Omega_w\setminus \Omega_{w,b}$. On the one hand, H\"older inequality and \eqref{energy-est} yield
	\begin{equation}\label{est-lift-bound}\begin{aligned}
	&\left|	\int_{\Omega_{w} \setminus  \Omega_{w,b}} u\cdot \nabla u  \cdot w + \mu\int_{\Omega_{w} \setminus  \Omega_{w,b}} \nabla u  : \nabla w \right|\\[2pt]&\leq C \|u\|^2_{H^1(\Omega_h}\|w\|_{L^\infty(\Omega_{w} \setminus  \Omega_{w,b})}+\mu\| \nabla u\|_{L^2(\Omega_h)}\| \nabla w \|_{L^2(\Omega_{w} \setminus  \Omega_{w,b})} \\[2pt] &\leq C( \|u\|^2_{H^1(\Omega_h)}+\| u\|_{H^1(\Omega_h)})
	\leq C\lambda\\[2pt]
	\end{aligned}
	\end{equation} for $\lambda\in[0,\Lambda_0]$, using that $w$ and its derivatives are uniformly bounded with respect to $h$ in $\Omega_{w} \setminus  \Omega_{w,b}$.
	On the other hand, since $w\equiv0$ in $\Omega^0_{w, b}:=[-2\tau,2\tau]\times [-H,h-\delta_b-\tfrac{\eps_b(h)}{2}]$ and  $u_{|_{\Gamma_b}}=0$, Poincaré inequality for $u$ in $\Omega_{w, b}\cup \Omega^0_{w, b}$, the H\"older inequality and \eqref{energy-est} yield
	\begin{equation}\label{est-lift-bound-b}\begin{aligned}
&\left|	\int_{\Omega_{w,b}} u\cdot \nabla u  \cdot w \right| =\left|\int_{\Omega_{w,b}\cup \Omega^0_{w,b}} u\cdot \nabla u  \cdot w \right|  \\[2pt]&\leq      \eps_b(h) \|\nabla u\|^2_{L^2(\Omega_{w,b}\cup \Omega^0_{w,b})} \|w\|_{L^\infty(\Omega_{w,b})}
	\leq C\| u\|^2_{H^1(\Omega_h)} \leq C\lambda
	\end{aligned}
	\end{equation}and
	\begin{equation}\label{est-lift-blow}
	\left|\int_{\Omega_{w,b}} \nabla u  : \nabla w\right| \leq \|  u\|_{H^1(\Omega_h)}\|\nabla w\|_{L^2(\Omega_{w,b})}\leq C(\eps_b(h))^{-3/2}\lambda,\\[2pt]
	\end{equation}for $\lambda\in[0,\Lambda_0]$, using that  $\|w\|_{L^\infty(\Omega_{w,b})}\leq C (\eps_b(h))^{-1}$ and $\|\nabla w\|_{L^2(\Omega_{w,b})}\leq C (\eps_b(h))^{-3/2}$ for $\eps_b(h) $ close to zero, due to \eqref{w-bounds-bot}.\\
	Putting together \eqref{est-lift-bound}-\eqref{est-lift-blow}, then there exists $\eta_b>0$ sufficiently small such that, for any $(\lambda,h)\in [0,\Lambda_0]\times (-H+\delta_b, -H+\delta_b+\eta_b),$
	\begin{equation}\label{est-lift-3/2}
	|\mathcal{L}(\lambda, h)|\leq C (\eps_b(h))^{-3/2} \lambda.
	\end{equation}
	We remark that the same blow-up rate in \eqref{est-lift-3/2} could be obtained without taking advantage of Poincaré inequality in \eqref{est-lift-bound-b} but using directly $u\in H^1\subset L^4$.  This idea, however, will be crucial to obtain a better blow-up rate for the lift in the case when the body is close to $\Gamma_t$, that we now analyze.\\[2pt]
	{\it $B_h$ close  to $\Gamma_t$.}  We consider the case when $h$ is close to $H-\delta_t$, hence $\eps_t(h)$ is close to zero. Analogously to what done in the previous case, we split the integrals over the regions $\Omega_{w,t}$, which is shrinking as $\eps_t(h)$ goes to zero, and $\Omega_w\setminus\Omega_{w,t}$. On the one hand, H\"older inequality yields
	\begin{equation*}\begin{aligned}
	&\left|	\int_{\Omega_w\setminus\Omega_{w,t}} u\cdot \nabla u \cdot w + \mu\int_{\Omega_w\setminus\Omega_{w,t}} \nabla u  : \nabla w\right| \\[2pt]&\leq C \|u\|^2_{H^1(\Omega_h)}\|w\|_{L^\infty(\Omega_w\setminus\Omega_{w,t})}+\mu\| \nabla u\|_{L^2(\Omega_h)}\| \nabla w \|_{L^2(\Omega_w\setminus\Omega_{w,t})} \\[2pt] &\leq C( \|u\|^2_{H^1(\Omega_h)}+\| u\|_{H^1(\Omega_h)})\\[2pt]
	\end{aligned}
	\end{equation*}using that $w$ and its derivatives are uniformly bounded with respect to $h$ in $\Omega_{w} \setminus  \Omega_{w,t}$. On the other hand, since $w\equiv0$ in $\Omega^0_{w, t}:=[-2\tau,2\tau]\times [h+\delta_t+\tfrac{\eps_t(h)}{2}, H]$ and $u=v+s$ with  $v_{|_{\Gamma_t}}=0$, Poincaré inequality for $v$ in $\Omega_{w,t}\cup \Omega^0_{w,t}$ and H\"older inequality yield
	\begin{equation*}\begin{aligned}
			&\left|\int_{\Omega_{w,t}} u\cdot \nabla u  \cdot w \right|=  \left|\int_{\Omega_{w,t}\cup \Omega^0_{w,t}} v\cdot \nabla u  \cdot w + \int_{\Omega_{w,t}\cup \Omega^0_{w,t}} s\cdot \nabla u  \cdot w\right|
			\\[2pt]
			&\leq \eps_t(h) \|\nabla v\|_{L^2(\Omega_h)} \|\nabla u\|_{L^2(\Omega_h)}\|w\|_{L^\infty(\Omega_{w,t})} + \|s\|_{L^2(\Omega_h)}\|\nabla u\|_{L^2(\Omega_h)}\|w\|_{L^\infty(\Omega_{w,t})}\\[2pt]
			&\leq C\|\nabla v \|_{L^2(\Omega_h)}\| u \|_{H^1(\Omega_h)} + C\|s\|_{L^2(\Omega_h)}\|u\|_{H^1(\Omega_h)}(\eps_t(h))^{-1}\\[2pt]
		\end{aligned}
	\end{equation*}
	and
	\begin{equation*}
		\left|\int_{\Omega_{w,t}} \nabla u : \nabla w\right|\leq \| u \|_{H^1(\Omega_h)}\|\nabla w\|_{L^2(\Omega_{w,t})}\leq \| u \|_{H^1(\Omega_h)}(\eps_t(h))^{-3/2},\\[2pt]
	\end{equation*}using that $\|w\|_{L^\infty(\Omega_{w,t})}\leq C (\eps_t(h))^{-1}$ and $\|\nabla w\|_{L^2(\Omega_{w,t})}\leq C (\eps_t(h))^{-3/2}$ for $\eps_t(h)$ close to zero,  due to \eqref{w-bounds-top}.
	Now we shall distinguish the cases $U=1$ and $U=0$. When $U=1$, using \eqref{energy-est}, \eqref{solext-eps1} and \eqref{energy-est-esp-v} we obtain, for $\lambda\in[0,\Lambda_0]$,
	\begin{equation}\label{est-lift-bound-top}\begin{aligned}
	&\left|\int_{\Omega_w\setminus\Omega_{w,t}} u\cdot \nabla u \cdot w + \mu \int_{\Omega_w\setminus\Omega_{w,t}} \nabla u  : \nabla w \right|
	\leq C(\eps_t(h))^{-3}\lambda
	\end{aligned}
	\end{equation}
and
	 \begin{equation}\label{est-lift-blow-top}
	 	\left|	\int_{\Omega_{w,t}} u\cdot \nabla u  \cdot w \right|\leq C(\eps_t(h))^{-3} \lambda, \quad
	 	\left|	\int_{\Omega_{w,t}} \nabla u : \nabla w  \right|\leq C (\eps_t(h))^{-3}\lambda.
	 \end{equation}
When $U=0$, using \eqref{energy-est} and $\eqref{s-bounded}$, we obtain, for $\lambda\in[0,\Lambda_0]$,
	\begin{equation}
	\left|\int_{\Omega_w\setminus\Omega_{w,t}} u\cdot \nabla u  \cdot w + \mu\int_{\Omega_w\setminus\Omega_{w,t}} \nabla u  : \nabla w \right|
	\leq C\lambda
	\end{equation}
and
	\begin{equation}\label{est-lift-blow-top2}
\left|	\int_{\Omega_{w,t}} u\cdot \nabla u  \cdot w\right| \leq C(\eps_t(h))^{-1}\lambda, \ \ \
\left|	\int_{\Omega_{w,t}} \nabla u : \nabla w \right| \leq C(\eps_t(h))^{-3/2}\lambda.\\[2pt]
	\end{equation}
	Putting together \eqref{est-lift-bound-top}-\eqref{est-lift-blow-top2}, then there exists $\eta_t>0$ sufficiently small such that, for $(\lambda,h)\in [0,\Lambda_0]\times (H-\delta_t-\eta_t, H-\delta_t),$
	\begin{equation}\label{est-lift-3}
	|\mathcal{L}(\lambda, h)|\leq C\max\{(\eps_t(h))^{-3/2},U (\eps_t(h))^{-3}\}\lambda.\\[5pt]
	\end{equation}For $h\in [-H+\delta_b + \eta_b, H-\delta_t-\eta_t]$,  $\eps_b(h)$ and $\eps_t(h)$ are uniformly bounded from below with respect to $h$. Therefore, by combining \eqref{est-lift-3/2} and \eqref{est-lift-3}, there exists $C>0$ independent of $h$ such that, for any $(\lambda,h)\in [0,\Lambda_0]\times (-H+\delta_b, H-\delta_t)$,
		\begin{equation*}
	|\mathcal{L}(\lambda, h)|\leq C ((\eps_b(h))^{-3/2}+ \max\{(\eps_t(h))^{-3/2},U (\eps_t(h))^{-3} \})\lambda.
	\end{equation*}

\section{Proof of Theorem \ref{main-theo}}\label{sec-proofMain}

\subsection{Continuity and monotonicity of the global force}\label{cont-mono-phi}
In Section \ref{sec-FSI} we have defined the lift $\mathcal{L}(\lambda,h)$ as a possibly multi-valued function of $(\lambda, h)\in  [0,+\infty)\times(-H+\delta_b,H-\delta_t)$.
Let $f$ be the restoring force satisfying \eqref{f}-\eqref{f-prop}. Then, the global force acting on $B_h$ is
the function $\phi:  [0,+\infty)\times(-H+\delta_b,H-\delta_t) \rightarrow \mathbb{R}$ defined by
\begin{equation}\label{phi}
\phi(\lambda, h)=f(h) - \mathcal{L}(\lambda, h).\end{equation}
We first focus on the $\lambda$-dependence by maintaining $h$ fixed and we
prove the Lipschitz-continuity of the map $\lambda \mapsto\phi(\lambda, h)$.

\begin{proposition}\label{phi-lip}
 Let $\overline{h}=H-\max\{\delta_b,\delta_t\}$.
  There exist $\overline{\lambda}>0$ and $h^*\in (0, \overline{h})$ such that $\lambda\mapsto\phi(\lambda, h)$ is Lipschitz continuous in $ [0,\overline{\lambda})$ for all $h\in [-h^*, h^*]$.
\end{proposition}

\begin{proof}To begin, let us take $\overline{\lambda}$ and $h^*$ sufficiently small so that Theorem \ref{strong-exiuni} guarantees the uniqueness for \eqref{BVprob} whenever $\lambda < \overline{\lambda}$ and $|h|\leq h^*$ (see Figure \ref{Lambdone}). Hence, $\mathcal{L}(\lambda, h)$ is a one-valued function on $[0,\overline{\lambda})\times [-h^*, h^*]$. Since $f$ does not depend on $\lambda$ we only need to show that $\lambda\mapsto \mathcal{L}(\lambda,h)$ is Lipschitz continuous in a neighborhood
of $\lambda=0$, possibly smaller than $[0, \overline{\lambda})$.\\ For $\lambda_1, \lambda_2\in [0,\overline{\lambda})$ consider, respectively,
the solutions $(u(\lambda_1), p(\lambda_1))$ and $(u(\lambda_2), p(\lambda_2))$ to \eqref{BVprob}. Let
\begin{equation}\label{defdeltaup}
v:= u(\lambda_1) - u(\lambda_2), \qquad q:= p(\lambda_1) - p(\lambda_2),\\[2pt]
\end{equation}
so that $(v,q)$ satisfies
\begin{equation}\label{BVprob-Lambda12}
\begin{aligned}
&-\mu\Delta v+v\cdot\nabla v+\nabla q =-v\cdot\nabla u(\lambda_2)-u(\lambda_2)\cdot \nabla v,\qquad\nabla\cdot v=0\quad\mbox{in}\quad \Omega_h,\\
&v_{|_{\Gamma_t}}=(\lambda_{ 1}\!-\!\lambda_{ 2} )Ue_1,\
v_{|_{\Gamma_l}}=(\lambda_{ 1}\!-\!\lambda_{ 2} )V_{\rm in}(x_2)e_1,\
v_{|_{\Gamma_r}}=(\lambda_{ 1}\!-\!\lambda_{ 2} )V_{\rm out}(x_2)e_1,\\
&v_{|_{\partial B_h}}=v_{|_{\Gamma_b}}=0.\\[2pt]
\end{aligned}
\end{equation}
Let $v_\lambda:=v-s_\lambda$, where $s_\lambda\in W^{1,\infty}(\Omega_h)\cap H^2(\Omega_h)$ is a solenoidal extension of $v$ that can be constructed as $s$ in \eqref{solext} and, hence, it satisfies the estimates \eqref{implicit-solext-est}, namely
\begin{equation}\label{solextdiff-est}
\begin{array}{rl}
\|\nabla s_\lambda\|_{L^2(\Omega_h)}\leq C_h|\lambda_1-\lambda_2|,\quad & \|\Delta s_\lambda\|_{L^2(\Omega_h)} \leq C_h|\lambda_1-\lambda_2|,\\	\|s_\lambda\|_{L^\infty(\Omega_h)}\le C_h|\lambda_1-\lambda_2|,\quad & \|s_\lambda \cdot \nabla s_\lambda \|_{L^2(\Omega_h)}\leq C_h|\lambda_{ 1}-\lambda_{ 2} |^2.\\[2pt]
\end{array}
\end{equation}
We then rewrite \eqref{BVprob-Lambda12} as
\begin{equation}\label{tecnica}
-\mu\Delta v_\lambda + v_\lambda \cdot \nabla v_\lambda + \nabla q =g,\quad \nabla \cdot v_\lambda=0\quad \mbox{in}\quad \Omega_h,
\qquad {v_\lambda}_{|_{\partial \Omega_h}}=0,\\[-5pt]
\end{equation}
where
$$g:=\mu\Delta s_\lambda-v\cdot\nabla(u(\lambda_2)+s_\lambda)-u(\lambda_2)\cdot\nabla v+s_\lambda\cdot\nabla s_\lambda-s_\lambda\cdot\nabla v.$$
From Theorem \ref{strong-exiuni} we know that $v,u(\lambda_2)\in H^2(\Omega_h)\hookrightarrow L^\infty(\Omega_h)$, so that $g\in L^2(\Omega_h)$. Moreover,

\begin{align*}
\|g\|_{L^2(\Omega_h)} &\le \mu\|\Delta s_\lambda\|_{L^2(\Omega_h)} +\big(\|\nabla u(\lambda_2)\|_{L^2(\Omega_h)}
+\|\nabla s_\lambda\|_{L^2(\Omega_h)}\big)\|v\|_{L^\infty(\Omega_h)}\\
 & \quad +\|u(\lambda_2)\|_{L^\infty(\Omega_h)}\|\nabla v\|_{L^2(\Omega_h)}\!+\!\|s_\lambda\!\cdot\!\nabla s_\lambda\|_{L^2(\Omega_h)}
 \!+\!\|s_\lambda\|_{L^\infty(\Omega_h)}\|\nabla v\|_{L^2(\Omega_h)}\\[2pt]
 &\le C_h|\lambda_1-\lambda_2| +C_h(\lambda_2+|\lambda_1-\lambda_2|\big)\|v\|_{H^2(\Omega_h)}\\
 & \quad  +C_h\lambda_2\|v\|_{H^2(\Omega_h)}+C_h|\lambda_{ 1}-\lambda_{ 2} |^2
 +C_h|\lambda_{ 1}-\lambda_{ 2} |\cdot\|v\|_{H^2(\Omega_h)},
\end{align*}
where we used H\"older inequality (first step), the estimates \eqref{energy-est}-\eqref{reg-est}-\eqref{solextdiff-est} and the
embeddings $H^2\hookrightarrow H^1,L^\infty$ (second step).
Thus, by extending the solution as in the proof of Theorem  \ref{strong-exiuni}, recalling \cite{MurSim74}, and applying  \cite[Theorem IV.5.1]{Galdi-steady} to \eqref{tecnica}, we obtain
\begin{equation}\label{vLambda-ineq}
\|v_\lambda\|_{H^2(\Omega_h)}+\|q\|_{H^1(\Omega_h)}\le C_h|\lambda_1-\lambda_2|
+C_h\big(\lambda_2+|\lambda_1-\lambda_2|\big)\|v\|_{H^2(\Omega_h)}.
\end{equation}
Hence, there exists a possibly smaller $\overline{\lambda}>0$ such that, if $\lambda_1, \lambda_2 \in[0,\overline{\lambda})$, the second
term in the right-hand side of \eqref{vLambda-ineq} can be absorbed in the left-hand side and
\begin{equation}\label{cont-dep}
\|v_\lambda\|_{H^2(\Omega_h)} + \|q\|_{H^1(\Omega_h)}  \leq C_h|\lambda_1-\lambda_2|,\\[2pt]
\end{equation}
for some $C_h>0$ also depending on $\overline{\lambda}$.
Since the lift \eqref{lift} is linear with respect to $u$ and $p$, we have
	\begin{equation*}
\mathcal{L}(\lambda_1,h)-\mathcal{L}(\lambda_2,h)= -e_2 \cdot \int_{\partial B_h}\mathbb{T}(v,q) n
	\end{equation*}
with $v$ and $q$ defined in \eqref{defdeltaup}.
Therefore, using the Trace Theorem and \eqref{cont-dep}, we infer that, for any $\lambda_1,\lambda_2\in[0,\overline{\lambda})$ and a
fixed $h\in [-h^*,h^*]$, we have
	\begin{equation*}
	\begin{aligned}
		|\mathcal{L}(\lambda_1, h)-\mathcal{L}(\lambda_2,h)|&\leq C_h\left( \|\nabla v\|_{L^1(\partial B_h)} + \|q\|_{L^1(\partial B_h)}\right)\\
		&\leq C_h\left( \|v\|_{H^2(\Omega_h)} + \| q\|_{H^1( \Omega_h)}\right)\leq C_h|\lambda_1-\lambda_2|.\\[2pt]
		\end{aligned}
	\end{equation*}
This shows that $\lambda\mapsto \mathcal{L}(\lambda, h)$ is Lipschitz continuous in $[0,\overline{\lambda})$ for all $h\in[-h^*,h^*]$.
\end{proof}
We now focus on the $h$-dependence of $\phi$ by maintaining $\lambda$ fixed. Although we prove a slightly stronger result, we state:

\begin{proposition}\label{phi-incre}
Let $\overline{h}=H-\max\{\delta_b,\delta_t\}$.
	There exist $h_0\in(0,h^*]$ and $\lambda_0\in(0, \overline{\lambda}]$ (see Proposition \ref{phi-lip}) such that $ h\mapsto \phi ( \lambda, h)$ is continuous and strictly increasing  in $[-h_0, h_0]$ for all $\lambda \in[0,\lambda_0)$.
\end{proposition}
\begin{proof}
Recall that $R=(-L,L)\times (-H,H)$. Let $ 0<r_1<r_2$  and $D_{r_i}(0)$ be the open disk centered at $(0,0)$ with radius $r_i$.  Choose $h_0\in (0, h^*)$ in such a way that $B_h\subset D_{r_1}(0)\subset D_{r_2}(0)\subset R$ whenever $|h|\leq h_0$; in later steps we may need to choose a possibly smaller $h_0$ that, however, we continue calling $h_0$. Let $\sigma\in  W^{2,\infty}(R, \mathbb{R}^2)$ be defined by
\begin{equation}\label{sigma}
\sigma(x_1,x_2)= F(|x|) e_2 ,
\end{equation}
with $ F\equiv 1$  in $[0,r_1]$, $F\equiv0$ in $[r_2, +\infty)$ and $F\in W^{2,\infty}(r_1,r_2)$ is the polynomial of third degree such that $F(r_1)=1$ and $F(r_2)=F'(r_1)=F'(r_2)=0$.
For $h\in [-h_0, h_0]$, with $h_0$ small,  we view the fluid domain $\Omega_h$ as a variation of $\Omega_0$ via the diffeomorphism $\mathrm{Id}+h\sigma$, that is, $$\Omega_h=(\mathrm{Id}+h\sigma)(\Omega_0).$$
In particular, $\partial B_h = \partial B_0 + he_2$ with unit outer normal vector $n(h)=n(0)\circ (\mathrm{Id}+ he_2)$.
 Let  $J(h)$ denote the Jacobian matrix of the diffeomorphism $\mathrm{Id} + h\sigma$, that is, $$J(h)=I + h \frac{F'(|x|)}{|x|}\left(\begin{matrix}
0&0\\
x_1&x_2
\end{matrix}\right) $$ with $I$ the $2\times 2 $ identity matrix.
Fixing $\lambda\in [0,\overline{\lambda})$, the lift in \eqref{lift} can be  written as
\begin{equation*}
\mathcal{L}(\lambda, h)=-e_2\cdot \int_{\partial B_0+he_2}\mathbb{T}(u(h), p(h))n(h)
\end{equation*}with $\mathbb{T}(u(h), p(h))=\mathbb{T}(u( \lambda, h), p( \lambda, h))$.
Letting
$$U( h )= u(h)\circ  (\mathrm{Id} + h\sigma), \quad P(h )= p(h )\circ (\mathrm{Id} + h\sigma)$$
with $\sigma$ as in \eqref{sigma},
we transform the moving boundary integral into a fixed boundary integral, namely
\begin{equation*}
\mathcal{L}(\lambda, h)=-e_2\cdot\int_{\partial B_0} \mathbb{T}(U(h),P(h)) (n(0)\circ (\mathrm{Id}+he_2)).
\end{equation*}
Note that $(U(0), P(0))=(u(0), p(0))$. We now claim that
\begin{equation}\label{claim}\vspace{0.1em}
h\mapsto (U(h),P(h))\in H^2(\Omega_0) \times H^1(\Omega_0)\mbox{ belongs to }  C^1(-h_0, h_0).\vspace{0.1em}
\end{equation}
 To this end, let $M(h)=(J^{-1}(h))^T$ and  we rewrite \eqref{BVprob} as
\begin{equation*}
\begin{aligned}
-\mu\nabla \cdot (|\det J(h)|M^T(h) M(h)\nabla U(h))\quad &\\+ U(h)\cdot |\det J(h)|M(h)\nabla U(h)  + \nabla \cdot  (|\det J(h)| M(h)P(h) )  =0   \quad &\mbox{in} \quad \Omega_0,\\[5pt] \  |\det J(h)| M(h)\nabla \cdot U(h) =0   \quad &\mbox{in} \quad  \Omega_0,
\end{aligned}
\end{equation*}complemented with the same boundary conditions.
This can also be expressed as
\begin{equation}\label{implicit}
\mathcal{H}(h,U(h), P(h))=0
\end{equation}where $\mathcal{H} :(-h_0,h_0)\times H^2(\Omega_0)\times H^1(\Omega_0) \rightarrow L^2(\Omega_0)\times H^1(\Omega_0)$ is defined by $\mathcal{H}(h, \xi, \varpi)=(\mathcal{H}_1(h, \xi, \varpi), \mathcal{H}_2(h, \xi, \varpi))$ with
\begin{equation}\label{defH}
	\begin{aligned}
	\mathcal{H}_1(h, \xi, \varpi)= &-\mu\nabla \cdot (|\det J(h)|M^T(h) M(h)\nabla \xi)\quad \\&+ \xi \cdot |\det J(h)|M(h)\nabla \xi   + \nabla \cdot  (|\det J(h)| M(h)\varpi ),\\[5pt]
	\mathcal{H}_2(h, \xi, \varpi)= & \ |\det J(h)| M(h)\nabla \cdot \xi.
	\end{aligned}
\end{equation}Due to the expression $\eqref{sigma}$, we are able to compute  $|\det J(h)|M(h)$ and $|\det J(h)|M^T(h)M(h)$ explicitly at second order for $h\rightarrow 0$. In fact,
\begin{equation*}\begin{aligned}
&|\det J(h)|= 1 + h \frac{F'(|x|)}{|x|}x_2,\\[5pt]& M(h)= I +\frac{h}{|\det J(h)|} \frac{F'(|x|)}{|x|}\left(\begin{matrix}
0 & -x_1\\0 & -x_2
\end{matrix}\right) =I +h \frac{F'(|x|)}{|x|}\left(\begin{matrix}
0 & -x_1\\0 & -x_2
\end{matrix}\right) + O(h^2)
\end{aligned}
\end{equation*}yield
\begin{equation}\label{det-M}\begin{aligned}
&	|\det J(h)|M(h)= I + h \frac{F'(|x|)}{|x|} \left(\begin{matrix}x_2&-x_1\\0&0\end{matrix}\right)=: I + h R_0,\\[5pt]
& |\det J(h)| M^T(h)M(h)\\&= I + h \frac{F'(|x|)}{|x|} \left(\begin{matrix}
x_2 & -x_1\\-x_1& -x_2
\end{matrix}\right)+  h^2 (F'(|x|))^2\left(\begin{matrix}
0&0\\0&1
\end{matrix}\right) +O(h^3)\\&=: I + h R_1 + h^2 R_2 +O(h^3),\\[5pt]
\end{aligned}
\end{equation}where $O(h^3)$ contains terms having at least third order with respect to $h$ as $h\rightarrow 0$. Note that the expression of $|\det J(h)|M(h)$ in \eqref{det-M} is exact and obtained without any Taylor expansion for $h\rightarrow 0$.
We have that $\mathcal{H}$ is $C^1$ in a neighborhood of $(0, U(0), P(0))$ since the mappings $h\mapsto \det J(h)$ and $h\mapsto M(h)$ are $C^1(-h_0,h_0)$  with values in $C^1(R, \mathbb{R}^4)$.\\
For $h\in (-h_0, h_0)$, we consider the linearized operator $\Upsilon= D_{(\xi, \varpi)}\mathcal{H} (h, U(h), P(h))$ defined through the Jacobian matrix of $\mathcal{H}$. For any $$(\chi, \Pi)\in \mathcal{X}\times \mathcal{Y}:=(H^2(\Omega_0)\cap H^1_0(\Omega_0))\times (H^1(\Omega_0)\cap L^2_0(\Omega_0)),$$ we have $	\Upsilon(\chi, \Pi)=( \Upsilon_1(\chi, \Pi),\Upsilon_2(\chi, \Pi) )$ with
\begin{equation*}\begin{aligned}
\Upsilon_1( \chi, \Pi)= &-\mu\nabla \cdot (|\det J(h)|M^T(h) M(h)\nabla \chi) + \chi \cdot |\det J(h)|M(h)\nabla U(h) \\& + U(h) \cdot |\det J(h)|M(h)\nabla \chi+ \nabla \cdot  (|\det J(h)| M(h)\Pi),\\[5pt]
\Upsilon_2( \chi, \Pi)= & \ |\det J(h)| M(h)\nabla \cdot \chi.\\[5pt]
\end{aligned}
\end{equation*}
The linear operator $\Upsilon$ is bounded from $\mathcal{X}\times \mathcal{Y}$ into $L^2(\Omega_0)\times \mathcal{Y}$. To show that $\Upsilon$ is an isomorphism,  given $(\varphi_1, \varphi_2)\in L^2(\Omega_0)\times \mathcal{Y}$, we have to prove that there exists a unique solution $(\chi, \Pi)\in \mathcal{X}\times \mathcal{Y}$ to
\begin{equation*}\begin{aligned}
&-\mu \Delta \chi + \chi \cdot \nabla U(h)+ U(h) \cdot \nabla \chi + \nabla \Pi \\&+ h \left(-\mu \nabla \cdot  R_1 \nabla \chi +\chi \cdot R_0 \nabla U(h) +  U(h)\cdot  R_0\nabla \chi  + \nabla \cdot (R_0 \Pi) \right)\\&
 -h^2 \mu \nabla \cdot R_2 \nabla \chi + O(h^3)= \varphi_1\quad &\mbox{in}  \ \Omega_0, \\[5pt]&\ \nabla\cdot \chi + h R_0 \nabla \cdot \chi  = \varphi_2 \quad &\mbox{in} \  \Omega_0.
\end{aligned}
\end{equation*}
 This linear elliptic problem admits a unique solution provided that
\begin{equation*}
|h|<h_0\qquad \mbox{and} \quad 	\|U(h)\|_{H^2(\Omega_0)}< r
\end{equation*}for $h_0, r>0$ small enough. For $|h|<h_0$ and $\sigma$ as in \eqref{sigma}, we have \begin{equation}\label{controlU-u}\begin{aligned}
 c\|u(h)\|_{H^2(\Omega_h)}\leq &\|U(h)\|_{H^2(\Omega_0)}\leq C\|u(h)\|_{H^2(\Omega_h)},\\[2pt] c\|p(h)\|_{H^1(\Omega_h)}\leq &\|P(h)\|_{H^1(\Omega_0)}\leq C\|p(h)\|_{H^1(\Omega_h)},
 \end{aligned}
\end{equation} with constants $0<c\leq C$ independent of $h$. Then, by taking $\lambda \in [0,\overline{\lambda})$, the bound \eqref{reg-est}, where the constant $C_h$ is uniformly bounded for $|h|<h_0$,
yields the needed smallness condition for $U(h)$, so that $\Upsilon$ is an isomorphism. Therefore, by applying the Implicit Function Theorem to \eqref{implicit}, we conclude \eqref{claim}.\\
Moreover, the derivatives $U'(h)$ and $P'(h)$, whose existence follows from \eqref{claim}, satisfy
\begin{equation}\label{eq-deriv}
	\Upsilon(U'(h), P'(h))= - \partial_h \mathcal{H} (h,U(h), P(h)).
\end{equation}From \eqref{det-M}, we know that for any $h$ (resp.  $h\rightarrow 0$)
\begin{equation*}
	\frac{d}{dh}(|\det J(h)| M(h))=R_0, \quad 	\frac{d}{dh}(|\det J(h)| M^T(h)M(h))=R_1+2hR_2 + O(h^2).
\end{equation*}
Then, recalling the definition \eqref{defH}, \eqref{eq-deriv} and the fact that $\Upsilon$ is an isomorphism imply that $(U'(h), P'(h))$ is uniquely determined by the linear elliptic problem
\begin{equation}\label{dotu-eq}
\begin{aligned}
&-\mu\Delta U'(h) +U'(h)\cdot \nabla U(h) +U(h)\cdot \nabla U'(h)+ \nabla  P'(h) \\&= S_0(U(h), P(h)) + hS_1(U'(h), P'(h), U(h))+O(h^2)   &\mbox{in} \ \Omega_0, \\[5pt]
&\nabla \cdot  U'(h)  = -R_0 \nabla \cdot U(h) - h R_0 \nabla \cdot U'(h)  &\mbox{in} \ \Omega_0,\\[5pt]
&(U'(h), P'(h))\in \mathcal{X}\times \mathcal{Y},
\end{aligned}
\end{equation}
with
\begin{equation*}\begin{aligned}
S_0(U(h), P(h))=&\  \mu \nabla \cdot R_1 \nabla U(h) - U(h)\cdot R_0 U(h) - \nabla \cdot (R_0 P(h)), \\[5pt]
S_1(U'(h), P'(h), U(h))=&\ \mu \nabla \cdot (R_1 \nabla U'(h) + 2R_2\nabla U(h))- U'(h)\cdot R_0 \nabla U(h)\\&- U(h)\cdot R_0 \nabla U'(h) - \nabla \cdot (R_0 P'(h)).
\end{aligned}
\end{equation*}
For $h\in (-h_0,h_0)$, with $h_0$ small, we have
\begin{equation*}\begin{aligned}
&\|U'(h)\|_{H^2(\Omega_0)} + \|P'(h)\|_{H^1(\Omega_0)} \\ &\leq C( \|U'(h)\!\cdot\! \nabla U(h) +U(h) \!\cdot\! \nabla U'(h) \!+\! S_0( U(h),P(h))\|_{L^2(\Omega_0)}  + \|R_0\nabla\! \cdot\! U(h)\|_{H^1(\Omega_0)}).\\[2pt]
\end{aligned}\end{equation*}
Since  $(U(h),P(h))\in H^2(\Omega_0)\times H^1(\Omega_0)$ due to \eqref{controlU-u} and Theorem \ref{strong-exiuni}, we bound the right-hand side of the above expression as
\begin{equation*}\begin{aligned}
&\|U'(h)\cdot \nabla U(h)+ U(h)\cdot \nabla U'(h)\|_{L^2(\Omega_0)}\leq C \|\nabla U'(h)\|_{L^2(\Omega_0)}\| U(h)\|_{H^2(\Omega_0)},\\[5pt]
&\|S_0( U(h), P(h))\|_{L^2(\Omega_0)}+ \|R_0\nabla \cdot U(h)\|_{H^1(\Omega_0)} \leq C(\|U(h)\|_{H^2(\Omega_0)} + \|P(h)\|_{H^1(\Omega_0)}),\\[2pt]
\end{aligned}
\end{equation*}where in the second inequality we used that $\sigma\in W^{2,\infty}(R, \mathbb{R}^2)$, see \eqref{sigma}.
Testing the first equation in \eqref{dotu-eq} with $U'(h)$, using \eqref{controlU-u} and \eqref{energy-est-gen}-\eqref{L2pressure-est}  yield
\begin{equation*}\begin{aligned}
\|\nabla U'(h)\|_{L^2(\Omega_0)}&\leq C(\| U(h)\|_{H^1(\Omega_0)}+\| U(h)\|^2_{H^1(\Omega_0)} + \| P(h)\|_{L^2(\Omega_0)} )\\[2pt] &\leq  C (\| U(h)\|_{H^1(\Omega_0)}+\| U(h)\|^2_{H^1(\Omega_0)}).
\end{aligned}
\end{equation*}
Summarizing, we obtain
\begin{equation}\label{derUP}\begin{aligned}
&\|U'(h)\|_{H^2(\Omega_0)} + \|P'(h)\|_{H^1(\Omega_0)}\\[2pt]&\leq C\left( \|U(h)\|_{H^2(\Omega_0)}(1+\| U(h)\|_{H^1(\Omega_0)}+\| U(h)\|^2_{H^1(\Omega_0)}) + \|P(h)\|_{H^1(\Omega_0)}\right)\\[2pt]&\leq C(\lambda + \lambda^3)\leq C\lambda
\end{aligned}
\end{equation}for any $\lambda\in[0,\overline{\lambda})$, where in the second inequality we used \eqref{controlU-u} and \eqref{energy-est}-\eqref{reg-est}.\\
Finally, we estimate the variation of the lift for small values of $h$, say $|h|< h_0$. By taking $h_1,h_2\in (-h_0,h_0)$, from the Trace Theorem we have
\begin{equation*}\begin{aligned}
&|	\mathcal{L}(\lambda, h_1) - 	\mathcal{L}(\lambda, h_2) |\\&= \left|\int_{\partial B_0} \mathbb{T}(U(h_1),P(h_1)) (n(0)\circ (\mathrm{Id}+h_1e_2) )- \mathbb{T}(U(h_2),P(h_2)) (n(0)\circ (\mathrm{Id}+h_2e_2) )\right|\\&
\leq \int_{\partial B_0}| \mathbb{T}(U(h_1),P(h_1)) - \mathbb{T}(U(h_2),P(h_2))| \\&\quad + \int_{\partial B_0} | \mathbb{T}(U(h_2),P(h_2))|  \cdot |n(0)\circ (\mathrm{Id} + h _1e_2)-n(0)\circ (\mathrm{Id} + h_2 e_2)|\\&
\leq C(\|U(h_1)- U(h_2)\|_{H^2(\Omega_0)}+ \|P(h_1)-P(h_2)\|_{H^1(\Omega_0)}) \\&\quad +C(\|U(h_2)\|_{H^2(\Omega_0)}+ \|P(h_2)\|_{H^1(\Omega_0)})|h_1-h_2|.\\[5pt]
\end{aligned}
\end{equation*}
 Then,  \eqref{derUP} and the Mean Value Theorem yield
\begin{equation*}\begin{aligned} &|	\mathcal{L}(\lambda, h_1) - \mathcal{L}(\lambda, h_2) |\\&\leq C\lambda |h_1-h_2| + C(\|u(h_2)\|_{H^2(\Omega_{h_2})}+ \|p(h_2)\|_{H^1(\Omega_{h_2})})|h_1-h_2|\leq C\lambda|h_1-h_2|\\[2pt]
\end{aligned}
\end{equation*}
using \eqref{controlU-u} and \eqref{reg-est} in $\Omega_{h_2}$. Then, the monotonicity property \eqref{f} ensures that, if $-h_0<h_2< h_1< h_0$,
\begin{equation*}\begin{aligned}
	\phi (\lambda, h_1) - \phi (\lambda, h_2) &= f( h_1) - f ( h_2) - \mathcal{L}(\lambda, h_1) + \mathcal{L}(\lambda,h_2) \geq (\gamma  -C\lambda)(h_1-h_2).\\[2pt]
\end{aligned}
\end{equation*}There exists $\lambda_0\in(0,\overline{\lambda}]$ such that
$\gamma  -C\lambda_0\geq \gamma/2$. Therefore, $h\mapsto \phi(\lambda, h)$ is continuous and strictly increasing in $[-h_0,h_0]$ (with a possible smaller $h_0$)  for all $\lambda\in [0,\lambda_0).$
\end{proof}

\subsection{Conclusion of the proof}\label{conclu-proof}

Let $(u(\lambda,h), p(\lambda, h))$ be a solution to \eqref{BVprob} and let $\phi(\lambda, h)$ be the corresponding global force in \eqref{phi}.
Then the triple $(u, p, h)$ is a solution to \eqref{FSI} if and only if
\begin{center}\vspace{0.1em}$(u(\lambda,h), p(\lambda, h))$ solves \eqref{BVprob} and $\phi(\lambda,h)=0.\vspace{0.1em}$\end{center}
Therefore, Theorem \ref{main-theo} follows once we prove:

\begin{proposition}
	Let $\phi$ be as in \eqref{phi} and $(\lambda_0,h_0)$ be as in Proposition \ref{phi-incre}. Then, there exist $\Lambda_1\in(0,\lambda_0]$ and a unique $\mathfrak{h}\in C^0[0,\Lambda_1)$ such that, for all $\lambda \in[0,\Lambda_1)$, $\phi(\lambda, h)=0$ if and only if $h=\mathfrak{h}(\lambda).$ Moreover, $\|\mathfrak{h}\|_{L^\infty(0, \Lambda_1)}\leq h_0$.
\end{proposition}
\begin{proof} We prove the result in two steps, namely by analyzing the behavior of $\phi$ in two different subregions of $[0,\lambda_0)\times(-H+\delta_b, H-\delta_t)$.\par
We start by considering the case when $|h|$ is close to 0.
Let again $\overline{h}=H-\max\{\delta_b,\delta_t\}$. We claim that
there exists $\widetilde{\lambda}\in (0,\lambda_0]$ and a unique $\mathfrak{h}\in C^0[0,\widetilde{\lambda})$ such that
\begin{equation}\label{claim1}
\forall (\lambda,h)\in[0, \widetilde{\lambda})\times[-h_0,h_0]\qquad\phi(\lambda, h)=0\ \Longleftrightarrow\ h=\mathfrak{h}(\lambda).
\end{equation}
To this end, we notice that Theorem \ref{strong-exiuni} implies that, when $\lambda=0$, the unique solution to \eqref{BVprob} is
$(u,p)=(0,0)$, regardless of the value of $h\in (-H+\delta_b, H-\delta_t)$.	Hence, $\phi(0,0)=0$.
Moreover, by Proposition \ref{phi-incre} we know that $h\mapsto\phi(0,h)$ is continuous and strictly increasing in $[-h_0,h_0]$.
These two facts imply that
\begin{equation}\label{(i)}
\phi(0,-h_0)<0<\phi(0,h_0).
\end{equation}
In turn, by Proposition \ref{phi-lip} we know that $\lambda\mapsto\phi (\lambda,h)$ is continuous in $[0,\overline{\lambda})$ for all $h\in[-h_0,h_0]$. By \eqref{(i)} and by compactness, we then infer that there exists
$\widetilde{\lambda}\in (0,\lambda_0]$ such that
\begin{equation}\label{(ii)}
\phi(\lambda,-h_0)<0<\phi(\lambda,h_0)\qquad\forall \lambda\in[0,\widetilde{\lambda})
\end{equation}and, by invoking again Proposition \ref{phi-incre}, that $h\mapsto\phi (\lambda,h)$ is continuous and strictly increasing in $[-h_0,h_0]$ for all $\lambda\in[0,\widetilde{\lambda})$.
Together with \eqref{(ii)}, this implies that for all $\lambda\in[0,\widetilde{\lambda})$ there exists a unique
	$\mathfrak{h}(\lambda)\in[-h_0,h_0]$ such that $\phi(\lambda,\mathfrak{h}(\lambda))=0$. This defines the function
$\lambda\mapsto\mathfrak{h}(\lambda)$ in the interval $[0,\widetilde{\lambda})$. Its continuity follows by the (separated) continuities proved in
Propositions \ref{phi-lip} and \ref{phi-incre}. The proof of \eqref{claim1} is so complete.\par
We now claim that there exists $\Lambda_1\in(0,\widetilde{\lambda}]$ such that
\begin{equation}\label{claim2}
\phi( \lambda,h)\neq 0\qquad \forall ( \lambda,h) \in [0, \Lambda_1)\times\Big[(-H+\delta_b, H-\delta_t)\setminus [-h_0,h_0]\Big].
\end{equation}
Recall that in this set $\phi(\lambda, h)$ may be multi-valued, see Theorem \ref{lift-theo}.
In order to prove \eqref{claim2}, from \eqref{f}-\eqref{f-prop} we know that there exists $K_0\in (0,K]$ such that
\begin{equation}\label{f-est}
\begin{aligned}
f(h) \leq  -K_0(\eps_b(h))^{-3/2}  \qquad &\mbox{for} \quad h\in(-H+ \delta_b,-h_0),\\[2pt]
	f(h)\geq  K_0\max \{ (\eps_t(h))^{-3/2}, U(\eps_t(h))^{-3}\} \qquad &\mbox{for} \quad  h\in (h_0, H-\delta_t),\\[5pt]
\end{aligned}
\end{equation}
while from Theorem \ref{lift-theo} there exists (a different) $C>0$ such that
\begin{equation}\label{lift-est}
\begin{aligned}
\mathcal{L}(\lambda,h) \geq  -C(\eps_b(h))^{-3/2} \ \lambda
\qquad &\mbox{for} \quad  h\in(-H+ \delta_b,-h_0), \\[2pt]
\mathcal{L}(\lambda,h) \leq C\max\{(\eps_t(h))^{-3/2}, U(\eps_t(h))^{-3}\} \ \lambda  \qquad &\mbox{for} \quad  h\in (h_0, H-\delta_t).\\[2pt]
\end{aligned}
\end{equation}
Gathering \eqref{f-est}-\eqref{lift-est} together yields
\begin{equation*}
\begin{aligned}
\phi(\lambda, h)\leq   ( -K_0 +C\lambda)(\eps_b(h))^{-3/2}
\qquad &\mbox{for} \quad  h\in(-H+ \delta_b,-h_0), \\[2pt]
\phi(\lambda,h)\geq  ( K_0-C\lambda )\max \{(\eps_t(h))^{-3/2}, U(\eps_t(h))^{-3} \}  \qquad &\mbox{for} \quad  h\in (h_0, H-\delta_t).\\[2pt]
\end{aligned}
\end{equation*}
Then, there exists $\Lambda_1\in (0,\widetilde{\lambda}]$ such that \eqref{claim2} holds and the statement of the proposition follows from \eqref{claim1} and \eqref{claim2}.
\end{proof}

\begin{remark}
In fact, the proof of \eqref{claim2} shows that if $\lambda>0$ is small, then
$$h_0<h<H-\delta_t\ \Longrightarrow\ \phi(\lambda,h)>0\quad\mbox{and}\quad-H+\delta_b<h<-h_0\ \Longrightarrow\ \phi(\lambda,h)<0.\\[2pt]$$
From a physical point of view, this means that, for small Reynolds numbers, the global force $\phi=\phi(\lambda, h)$ in \eqref{phi}
pushes downwards the body if $B_h$ is close to the upper boundary $\Gamma_t$, whereas it pushes the body upwards
if $B_h$ is close to the lower boundary $\Gamma_b$.
\end{remark}

\section{Symmetric configuration}\label{FSIsymm}

We consider here a symmetric framework for \eqref{FSI}, that is, when
\begin{equation*}
	(x_1,x_2)\in \partial B \iff (x_1,-x_2)\in \partial B
\end{equation*}
 and the boundary data are symmetric with respect to the line $x_2=0$. Therefore, the FSI problem \eqref{FSI} is modified on $\Gamma_b$ and reads
\begin{equation}\label{FSI-sym}
\begin{array}{cc}
-\mu\Delta u +u\cdot \nabla u + \nabla p = 0,\qquad\nabla \cdot u =0\quad\mbox{in}\quad \Omega_h\\
u_{|_{\partial B_h}}\!=0, \quad \!u_{|_{\Gamma_b}}\!= u_{|_{\Gamma_t}}\!=\lambda Ue_1,\quad
u_{|_{\Gamma_l}}\!=\lambda V_{\rm in}(x_2)e_1,\quad u_{|_{\Gamma_r}}\!=\lambda V_{\rm out}(x_2)e_1,
\\[5pt]
f(h)=-e_2\cdot \int_{\partial B_h} \mathbb{T}(u,p)n,\\[2pt]
\end{array}
\end{equation} with $\lambda\geq 0$, $U\in\{0,1\}$ (up to normalization). Here,
$V_{\rm in},V_{\rm out}\in W^{2,\infty}(-H,H)$ are now even functions satisfying
\begin{equation}\label{match-conds-sym}
V_{\rm in} (\pm H)=  V_{\rm out}(\pm H)=U,\qquad\int_{-H}^H V_{\rm in}(x_2) dx_2=\int_{-H}^H V_{\rm out}(x_2) dx_2.
\end{equation}
In this symmetric framework, $\delta_b=\delta_t=\delta$ and $h\in(-H+\delta, H-\delta).$ Then, we  prove that the unique curve $\mathfrak{h}(\lambda)$ found in Theorem \ref{main-theo} reduces to $\mathfrak{h}(\lambda)\equiv 0$, namely that the unique  equilibrium position is symmetric. Again, we expect this position to be stable, at least for small $\lambda.$

\begin{theorem}\label{symmtheo}
	Let $V_{\rm in}$, $V_{\rm out}\in W^{2,\infty}(-H,H)$ be even functions satisfying \eqref{match-conds-sym} and $ f\in C^0(-H+\delta, H-\delta)$ satisfying $f(0)=0$ and \eqref{f}-\eqref{f-prop} with $\delta_b=\delta_t=\delta$. There exists $\Lambda_1 >0$ such that for  $\lambda\in[0,\Lambda_1)$ the FSI problem \eqref{FSI-sym} admits a unique strong solution $(u(\lambda,h), p(\lambda,h),h)\in H^2(\Omega_h)\times H^1(\Omega_h)\times (-H+\delta, H-\delta)$ given by
	\begin{equation*}
	\big(u^0(\lambda,0), p^0(\lambda,0),0\big),
	\end{equation*}
	where $(u^0(\lambda,0), p^0(\lambda,0))$ is the unique solution to the first two lines in \eqref{FSI-sym} for $h=0$ and has the following symmetries:
	\begin{equation*}
	u^0_1(x_1,-x_2)=u^0_1(x_1, x_2), \quad u^0_2(x_1,-x_2)=-u^0_2(x_1, x_2), \quad p^0(x_1,-x_2)=p^0(x_1,x_2).
	\end{equation*}
\end{theorem}
\begin{proof}
	The first step is to obtain the counterpart of Theorem \ref{strong-exiuni}. The case $U=0$ is already included in the original statement. When $U=1$, we construct the cut-off functions $\zeta_l$ and $\zeta_r$ in a slightly different way with Figure \ref{cutoffsU=1} replaced by Figure \ref{cutoffs-sym}. We define the solenoidal extension as in \eqref{solext}, which satisfies the boundary conditions in \eqref{FSI-sym}.
	
		\begin{figure}[!h]
		\begin{center}
			\includegraphics[scale=0.13]{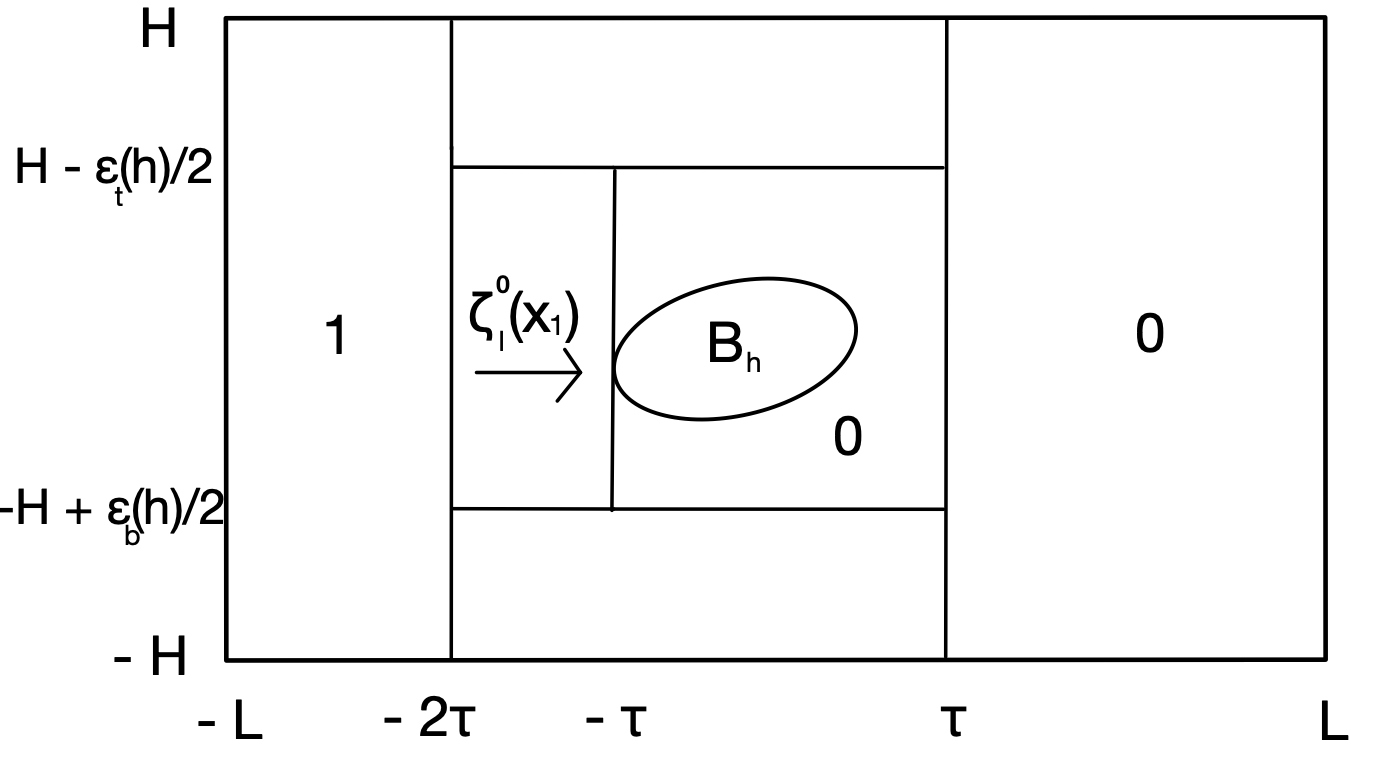}\includegraphics[scale=0.13]{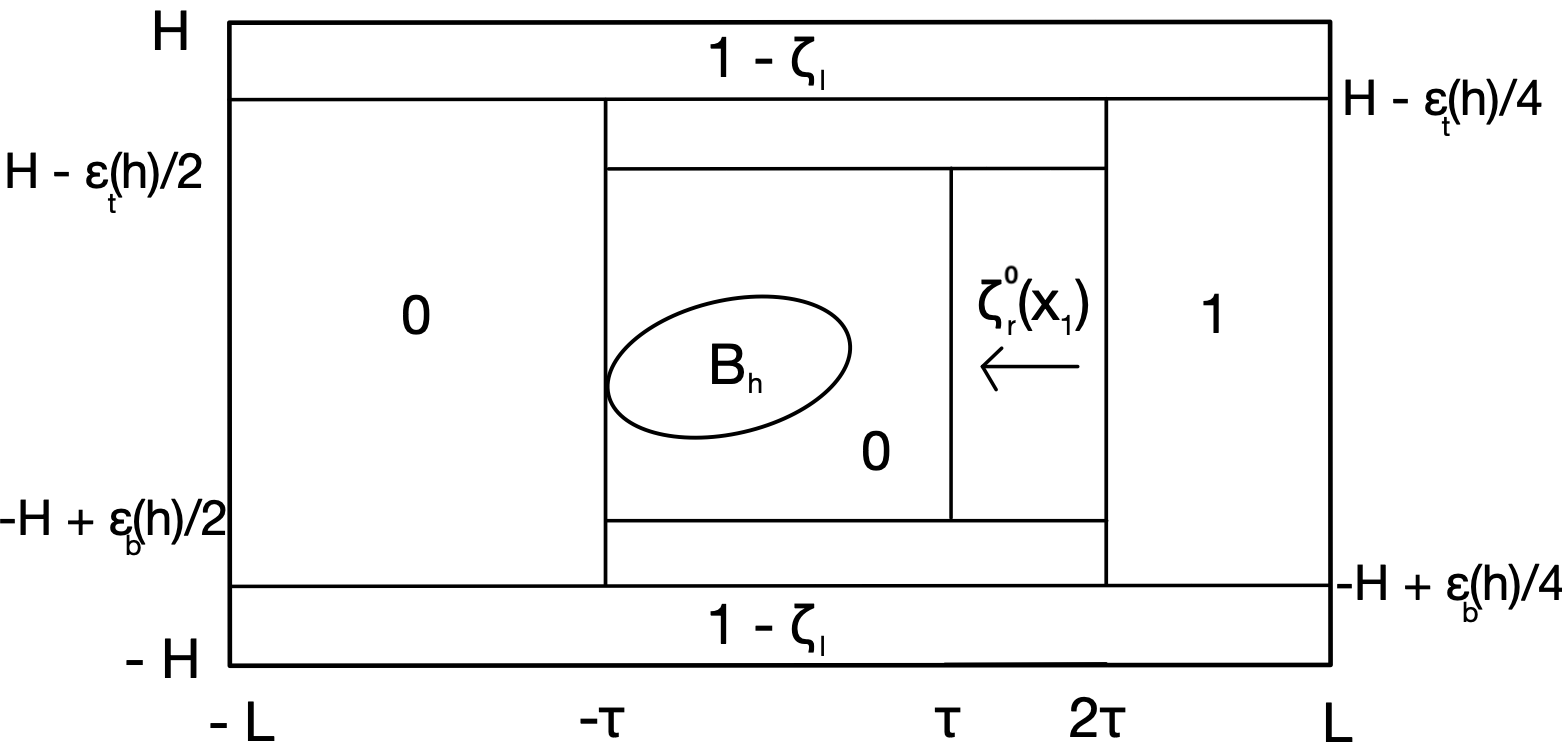}
			\caption{The cut-off functions $\zeta_l$ (left) and $\zeta_r$ (right) on $\overline{ R}$ when $U=1$ for the symmetric configuration.}
			\label{cutoffs-sym}
		\end{center}
	\end{figure}
	 With this construction the refined bound \eqref{energy-est} is replaced by
\begin{equation*}
	\| u\|_{H^1(\Omega_h)}  \leq C	((\eps_b(h))^{-3/2}+(\eps_t(h))^{-3/2})\lambda.\\[2pt]
\end{equation*}
	Hence, in both cases $U\in\{0,1 \}$, by arguing as in the proof of Theorem \ref{strong-exiuni}, we infer that
	there exists $\Lambda=\Lambda(h)>0$  such that for $\lambda\in[0,\Lambda(h))$ the solution $(u,p)$ to
	\begin{equation}\label{BVprob-sym}
	\begin{array}{cc}
	-\mu\Delta u+u\cdot \nabla u + \nabla p = 0,\qquad\nabla \cdot u =0\quad\mbox{in}\quad \Omega_h\\[2pt]
	u_{|_{\partial B_h}}\!=0, \quad \!u_{|_{\Gamma_b}}\!= u_{|_{\Gamma_t}}\!=\lambda Ue_1,\quad
	u_{|_{\Gamma_l}}\!=\lambda V_{\rm in}(x_2)e_1,\quad u_{|_{\Gamma_r}}\!=\lambda V_{\rm out}(x_2)e_1,\\[2pt]
	\end{array}
	\end{equation} is unique for any $h\in(-H+\delta,H-\delta)$. This proves the counterpart of Theorem \ref{strong-exiuni}.\\
	 In particular, for $h=0$ there exists a unique solution $(u^0,p^0)$ to \eqref{BVprob-sym} in $\Omega_0$. Since $\Omega_0$ is symmetric with respect to the line $x_2=0$, the couple
	$(u^*, p^*): \Omega_0 \rightarrow \mathbb{R}^2\times \mathbb{R}$ defined by
		\begin{equation*}\begin{aligned}
	u_1^*(x_1,x_2)= u^0_1(x_1, -x_2), \quad 	u_2^*(x_1,x_2)= -u^0_2(x_1, -x_2),\quad p^*(x_1, x_2)= p^0(x_1,-x_2),\\[2pt]
	\end{aligned}
	\end{equation*} also satisfies \eqref{BVprob-sym} for $h=0$ (see also \cite{GazSpe20}). Therefore, by uniqueness $(u^0, p^0)=(u^*, p^*)$ is also symmetric and, thanks to all these symmetries, we obtain
	\begin{equation*}\mathcal{L}(\lambda, 0)=-e_2\cdot\int_{\partial B_0} \mathbb{T}(u^0(\lambda, 0), p^0(\lambda, 0))n= 0,\\[2pt] \end{equation*}
which implies
\begin{equation}\label{phi0}
	\phi(\lambda,0)= f(0)=0 \qquad \mbox{for} \quad \lambda\in[0,\Lambda(h)).\\[2pt]
\end{equation} From Theorem \ref{main-theo} we know that there exist $\Lambda_1>0$ and a unique curve $\mathfrak{h}\in C^0[0,\Lambda_1)$ such that for $\lambda\in[0,\Lambda_1)$ the unique solution to \eqref{FSI-sym} is given by
$$(u(\lambda,\mathfrak{h}(\lambda)), p(\lambda,\mathfrak{h}(\lambda)),\mathfrak{h}(\lambda)).$$ Thanks to \eqref{phi0}, $\mathfrak{h}(\lambda)\equiv 0$ and this solution coincides with $(u^0(\lambda,0), p^0(\lambda,0),0)$.	
\end{proof}

\section{An application: equilibrium positions of the deck of a bridge}\label{suspension}

A suspension bridge is usually erected starting from the anchorages and the towers. Then the sustaining cables are installed
between the two couples of towers and the hangers are hooked to the cables. Once all these components are in position, they
furnish a stable working base from which the deck can be raised from floating barges. We refer to \cite[Section 15.23]{Podolny} for full details. The deck segments are put in position one aside the other (see Figure \ref{bridge}, left) and have the shape of rectangles while their
cross-section resembles to smoothened irregular hexagons (see Figure \ref{bridge}, right) that satisfy \eqref{B-W2infty}.

\begin{figure}[!h]
	\begin{center}
		\includegraphics[width=6cm]{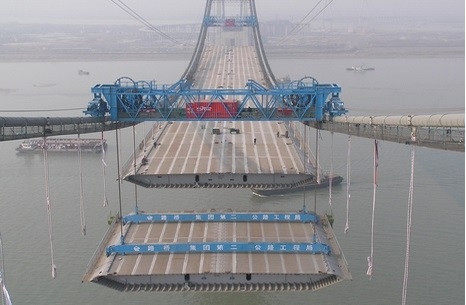}\quad\includegraphics[width=7cm]{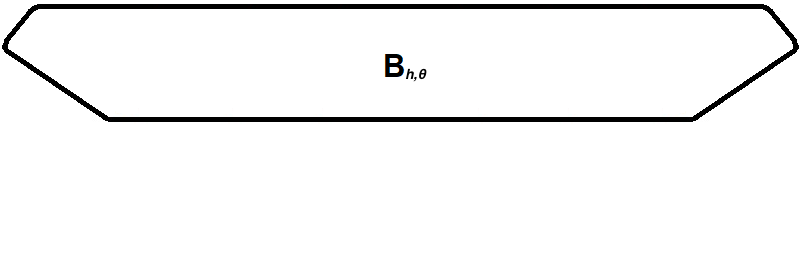}
		\caption{Left: erection of a suspension bridge. Right: sketch of a cross-section.}\label{bridge}
	\end{center}
\end{figure}

This cross-section $B$ plays the role of the obstacle in \eqref{BVprob} while $\Omega_h$ is the region filled by the air.
This region can be either be a virtual box around the deck of the bridge or a wind tunnel around a scaled model of the bridge.
In both cases, we may refer to inflow and outflow also as windward and leeward respectively:
$\lambda V_{\rm in}e_1$ represents the laminar horizontal windward while $\lambda V_{\rm out}e_1$ is the leeward.
Typically, the higher is the altitude the stronger is the wind. Therefore, in this application we consider specific laminar shear flows, which are the Couette flows. Thus, the inflow and outflow now read
\begin{equation}\label{couette}
	V_{\rm in}(x_2)= V_{\rm out}(x_2)=\frac{U}{2H}(x_2+H) \qquad \mbox{for}\qquad x_2\in[-H,H],
\end{equation}
and satisfy \eqref{match-conds}. The windward creates both vertical and torsional displacements of the deck. However, the cross-section of the suspension bridge is also subject to some elastic restoring forces tending to maintain the deck in its original position $B_0$.
These forces are of three different kinds. There is an upwards restoring force due to the elastic action of both the hangers and the sustaining cables of the bridge. The hangers behave as nonlinear springs which may slacken \cite[9-VI]{ammann} so that they have no downwards action and
they be nonsmooth. There is the weight of the deck which acts constantly downwards: this is why there is no odd requirement on the restoring force considered in the model. There is also a nonlinear resistance to both elastic bending and stretching of the whole deck for which $B$ merely represents a cross-section. Moreover, since the boundary of the channel $R$ is virtual and our physical model breaks down in case of collision of $B$ with $\partial R$, we require that there exists an ``unbounded force" preventing collisions.\\
Overall, the position of $B$ depends on both the displacement parameter $h$ and the angle of rotation $\theta$ with respect to the horizontal axis. With the addition of this second degree of freedom, we have
$B=B_{h,\theta}$ and $\Omega=\Omega_{h,\theta}$. A ``plastic'' regime leading to the collapse of the bridge is reached when $\theta=\pm \tfrac{\pi}{4}$ (see \cite{ammann}) since the sustaining cables of the bridge attain their maximum elastic tension.
The strong point of the analysis carried out in this paper is that it applies independently of the part of $\partial B$ closest to $\partial R$.
Therefore, for any $\theta\in(-\tfrac{\pi}{4}, \tfrac{\pi}{4})$, we can apply our general theory considering the family of bodies $B_{h,\theta}$ simply by adapting it to the rotating scenario. The only difference now is that,  when the body is free to rotate, the collision with $\Gamma_b$ and $\Gamma_t$ occurs at $h=-H+\delta_b(\theta)$ and $h= H-\delta_t(\theta)$, where  $\delta_b(\theta)$ and $\delta_t(\theta)$ are positive functions of $\theta$.
For $\theta=0$, $\delta_b(0)$ and $\delta_t(0)$ are as in \eqref{def-delta} while, for $\theta\neq 0$,
\begin{equation*}
	\delta_b(\theta):= -\min_{(x_1,x_2)\in \partial B_{0,\theta}}x_2>0, \qquad \delta_t(\theta):= \max_{(x_1,x_2)\in \partial B_{0,\theta}}x_2>0,
\end{equation*}both being independent of $h$.
 Due to the possible complicated shape of $B$, these functions are not easy to be determined explicitly. For this reason, we define the set of non-contact values of $(h,\theta)$ by
 \begin{equation}\label{admissible}
	A=\{(h,\theta)\in (-H,H)\times (-\tfrac{\pi}{4}, \tfrac{\pi}{4}) \ : \ B_{h,\theta} \subset R \}.
\end{equation}
Clearly, $(0,0) \in A$ and $(h,\theta)\in \partial A$ if and only if $B_{h,\theta}\cap\partial R \neq \emptyset.$ We assume that, for some $K>0$, $f\in C^0(A)$ satisfies
\begin{equation}\label{f-theta-prop}
\begin{aligned}
&\limsup\limits_{ d( B_{h,\theta}, \Gamma_b)\rightarrow 0} \ f(h,\theta)(d( B_{h,\theta}, \Gamma_b))^{3/2}\leq -K,\\[2pt]
&\liminf\limits_{d( B_{h,\theta}, \Gamma_t)\rightarrow 0} \ \frac{f(h,\theta)}{\max \{(d( B_{h,\theta}, \Gamma_t))^{-3/2},U(d( B_{h,\theta}, \Gamma_t))^{-3}\}}\geq K,\\[2pt]
\end{aligned}
\end{equation}
where
 $d( \cdot, \cdot)$ is the distance function. Assumption \eqref{f-theta-prop} generalizes \eqref{f-prop} taking into account the rotational degree of freedom. Moreover, we assume that  \begin{equation}\begin{aligned}\label{f-theta}
& \exists\gamma>0\quad\mbox{s.t.}\quad\frac{f(h_1,\theta)-f(h_2, \theta)}{h_1-h_2}\geq\gamma\qquad\forall (h_1,\theta),(h_2,\theta) \in A,\\[2pt]&
 f(0,0)=0, \qquad f(h,\theta)\theta>0 \quad \forall (h,\theta)\in A \quad \mbox{with} \quad \theta\neq 0.\\[2pt]
 \end{aligned}
 \end{equation}In fact, the second line in \eqref{f-theta} is not mathematically needed but, from a physical point of view, it states that the restoring force does not act at equilibrium and tends to maintain $B$ in an horizontal position.
   A straightforward consequence of Theorem \ref{main-theo}, in the case of the interaction between the wind and the deck of a suspension bridge, is the following:
\begin{corollary}
	Let $V_{\rm in}$, $V_{\rm out}$ be as in \eqref{couette} and $f\in C^0(A)$ satisfy \eqref{f-theta-prop}-\eqref{f-theta}. There exists $\Lambda_1>0$ and a unique $\mathfrak{h}\in C^0[0, \Lambda_1)$ such that, for $\lambda\in[0,\Lambda_1)$ and $\theta\in(-\tfrac{\pi}{4}, \tfrac{\pi}{4})$, the FSI problem \eqref{FSI} admits a unique solution
	$(u_\theta(\lambda,h), p_\theta(\lambda,h),h)\in H^2(\Omega_{h,\theta})\times H^1(\Omega_{h,\theta})\times (-H,H) $, with $(h,\theta)\in A$,
	given by
		\begin{equation*}(u_\theta(\lambda,\mathfrak{h}(\lambda)), p_\theta(\lambda,\mathfrak{h}(\lambda)),\mathfrak{h}(\lambda)).
	\end{equation*}
	Here, \eqref{FSI} is understood with $h$ replaced by the couple $(h, \theta)$.
\end{corollary}
The deck of a suspension bridge, in particular its cross-section, may have a nonsmooth boundary.
If $B$ is not  $W^{2,\infty} $ but it is only Lipschitzian, Theorem \ref{strong-exiuni}
ceases to hold and we only know that $(u,p)$ is a weak solution to \eqref{BVprob} so that \eqref{lift} does not hold in a ``strong" sense.
Indeed, since $u\in H^1(\Omega_h)$, see \eqref{energy-est}, we may rewrite the first equation in \eqref{BVprob} as
$-\mu\Delta u+\nabla p=f$ with $f\in L^q(\Omega_h)$ for all $q<2$. Hence, $f\in H^{-\epsilon}(\Omega_h)$ for any $\epsilon>0$. By
applying \cite[Theorem 7]{savare} we then deduce that $u\in H^{1+s}(\Omega_h)$ and $p\in H^s(\Omega_h)$ for all $s<1/2$ but, still,
this does not allow to consider the trace of $\mathbb{T}(u,p)$ as an integrable function over $\partial B_h$.
However, following \cite{GazSpe20} we may define the lift $L$ through a generalized formula. Indeed, from $u \in H^1(\Omega_h)$ we know that $\mathbb{T}(u,p)\in L^2(\Omega_h)$ and, since $\Omega_h$ is a bounded domain, $\mathbb{T}(u,p)\in L^{3/2}(\Omega_h)$. Moreover,
from the first equation in \eqref{BVprob} we obtain
$\nabla\cdot \mathbb{T}(u,p) \in L^{3/2}(\Omega_h)$. Therefore $\mathbb{T}(u,p)\in E_{3/2}(\Omega_h):=\{f\in L^{3/2}(\Omega_h) \ | \ \nabla \cdot f\in L^{3/2}(\Omega_h)\}$. By Theorem III.2.2 in \cite{Galdi-steady} we know that $\mathbb{T}(u,p)n_{|_{\partial\Omega_h}}\in W^{-2/3, 3/2}(\partial\Omega_h)$. Hence, if  $\partial B_h$ is Lipschitzian and $(u,p)$ is a weak solution to \eqref{BVprob}, then the lift exerted by the fluid over $B_h$ is
\begin{equation}\label{weaklift}
\mathcal{L}(\lambda,h)=-e_2 \ \cdot \langle \mathbb{T}(u,p)n,1\rangle_{\partial B_h},
\end{equation}
where $\langle\cdot,\cdot\rangle_{\partial B_h}$ denotes the duality pairing between $W^{-2/3,3/2}(\partial B_h)$ and $W^{2/3, 3}(\partial B_h)$.

\subsection*{Acknowledgments.}The authors warmly thank the anonymous referee for the careful proofreading and several useful remarks.\\
The authors were partially supported by the PRIN project {\it Direct and inverse problems for partial differential equations:
theoretical aspects and applications} and by the Gruppo Nazionale per l'Analisi Matematica, la
Probabilit\`a e le loro Applicazioni (GNAMPA) of the Istituto Nazionale di Alta Matematica (INdAM).\\
{\bf Data availability statement.} Data sharing not applicable to this article as no datasets were generated or analysed during the current study. There are no conflicts of interest.
\vspace{-0.6em}
\bibliographystyle{siam}
\bibliography{references}

\end{document}